\documentclass[11pt,letter]{article}
\usepackage[margin=1in]{geometry}

\usepackage[english]{babel}
\usepackage[utf8]{inputenc}
\usepackage[T1]{fontenc}
\usepackage{stmaryrd}
\usepackage{footmisc}

\usepackage{csquotes}

\usepackage{mathrsfs}
\usepackage{amsmath}
\usepackage{amssymb}
\usepackage{bbm}
\usepackage{graphicx}
\usepackage{stackengine}
\usepackage{amsthm}
\usepackage{subcaption}

\usepackage{thmtools}
\usepackage{thm-restate}

\usepackage[colorinlistoftodos]{todonotes}
\usepackage[colorlinks=true, allcolors=blue]{hyperref}

\usepackage{amsfonts}
\usepackage{cleveref}

\usepackage{soul}

\newtheorem{theorem}{Theorem}[section]

\newtheorem{lemma}[theorem]{Lemma}

\newtheorem{corollary}[theorem]{Corollary}

\newtheorem{claim}[theorem]{Claim}
\theoremstyle{remark}

\newcommand{\binuparrow}{\mathbin{\uparrow}}
\newcommand{\T}{\text{Top}}
\newcommand{\B}{\text{Bot}}
\newcommand{\F}{\mathscr{F}}
\renewcommand{\L}{\mathcal{L}}    
\newcommand{\set}[1]{\left\{#1\right\}}

\def\rddots#1{\cdot^{\cdot^{\cdot^{#1}}}}

\usepackage{tikz}
\usetikzlibrary{fit,positioning,calc}

\DeclareMathOperator{\dist}{dist}

\title{$k$-Leaf Powers Cannot be Characterized by a Finite Set of Forbidden Induced Subgraphs for $k \geq 5$}
\date{ }
\author{
    Max Dupr\'e la Tour\thanks{McGill University: \texttt{maxduprelatour@gmail.com}} \and 
       Manuel Lafond\thanks{Universit\'e de Sherbrooke: \texttt{manuel.lafond@usherbrooke.ca}} \and 
  Ndiam\'e Ndiaye\thanks{McGill University. \texttt{ndiame.ndiaye@mail.mcgill.ca}} \and 
    Adrian Vetta\thanks{McGill University. \texttt{adrian.vetta@mcgill.ca}}
}

\begin{document}
\maketitle

\begin{abstract}
A graph $G=(V,E)$ is a $k$-leaf power if there is a tree $T$ whose leaves are the vertices of $G$ with the property that 
a pair of leaves $u$ and $v$ induce an edge in $G$ if and only if they are distance at most $k$ apart in $T$.
For $k\le 4$, it is known that there exists a finite set $\F_k$ of graphs such that the class $\L(k)$ of $k$-leaf power graphs is characterized as the set of strongly chordal graphs that do not contain any graph in $\F_k$ as an induced subgraph. 
We prove no such characterization holds for $k\ge 5$.
That is, for any $k\ge 5$, there is no finite set $\F_k$ of graphs such that $\L(k)$ is equivalent to the set of strongly chordal graphs that do not contain as an induced subgraph any graph in $\F_k$.
\end{abstract}

\section{Introduction}\label{sec:intro}

A fundamental question in graph theory concerns whether or not a graph $G=(V,E)$ can be represented (or approximated) by a simpler graph, for instance a tree $T$, while preserving the desired information from the original graph.  
The pairwise distances of $G$ often need to be summarized into sparser structures, with 
notable examples including \emph{graph spanners}~\cite{cohen2020light,ahmed2020graph,filtser2021graph,le2022near} and \emph{distance emulators}~\cite{thorup2006spanners,chang2022almost,van2022fast} which respectively ask for a subgraph of $G$ or for another graph that approximates the distances of $G$.
If the distance information to preserve only concerns ``close together'' versus ``far apart'' then this can take the following form:
given a graph $G$ and an integer $k$, does there exists a tree $T$ whose leaves are the vertices of $G$, such that distinct vertices $u$ and $v$ are adjacent in $G$  {\em if and only if} the distance $d_T(u,v)$ from $u$ to $v$ in $T$ is at most $k$?  If the answer is affirmative then $G$ is dubbed a $k$-{\em leaf power} of $T$ (and $T$ is dubbed a $k$-{\em leaf root} of $G$). 

The study of $k$-leaf powers and roots were instigated by Nishimura, Ragde and Thilikos~\cite{NISHIMURA200269}. On the applied side, these graphs are of significant interest in the field of computational biology with respect to {\em phylogenetic trees}, which aim to explain the distance relationships observed on available data between species, genes, or other types of taxa.  Indeed, $k$-leaf powers can be used to represent and explain pairs of genes that underwent a bounded number of evolutionary events in their evolution~\cite{long2020exact,hellmuth2020generalized}, or that have conserved closely related biological functions during evolution~\cite{lafond2014orthology}.
On the theory side, despite their simplicity, several fundamental graph theoretic problems concerning $k$-leaf powers remain open.
The purpose of this research is to resolve one such long-standing 
open problem. Specifically, we 
prove that the class $\L(k)$ of $k$-leaf power graphs cannot be characterized via a finite set of forbidden induced subgraphs for $k\ge 5$. In contrast, for $k\le 4$ such finite characterizations were previously shown to exist~\cite{DOM2004, BRANDSTADT2006,Brandstadt2008-4LP}.

\subsection{Background}
Let $\L(k)$ denote the class of all $k$-leaf power graphs, for $k \geq 2$. The class of all leaf power graphs is then
denoted by $\L = \bigcup_k \L(k)$.
The literature on leaf power graphs has primarily focused on two major themes. One, obtaining graphical characterizations for both the class $\L$ and the classes $\L(k)$, for fixed values of $k$. 
Two, designing efficient algorithms to recognize graphs that belong to these classes. 

Let's begin with the former theme. Here  important roles are played by chordal and strongly chordal graphs.
A graph is \emph{chordal} if every cycle of length four or more has a \emph{chord}, an edge connecting two non-consecutive vertices of cycle. A graph is \emph{strongly chordal} if it is chordal {\bf and} all its even cycles of length $6$ or more have an \emph{odd chord}, a chord connecting two vertices an odd distance apart along the cycle.
Now, it is known that every graph in $\L$ and $\L(k)$ is \emph{strongly chordal}.\footnote{In particular they do not contain, as induced subgraphs, chordless cycles of length greater than three, nor {\em sun graphs}.} To see this, first note that a leaf power graph is an induced subgraph of a power of a tree. Second,
note that trees are strongly chordal, and taking powers and induced subgraphs both preserve this property~\cite{raychaudhuri1992powers}).  
However, the reciprocal is not true: there exist strongly chordal graphs that are not leaf powers. The first such example was discovered by Brandst\"adt et al.~\cite{BRANDSTADT2010897}. Subsequently, six additional examples were identified by Nevries and Rosenke~\cite{nevries2016towards} who conjectured that any strongly chordal graph not containing any of these seven graphs as an induced subgraph is a leaf power. 
% This conjecture, or the weaker version that there are only a finite number of obstructions, would imply the existence of a polynomial-time algorithm to determine if a given graph is a leaf power. 
However, a weaker version of this conjecture, that there are only a finite number (rather than seven) of obstructions was
was disproved by Lafond~\cite{Lafond2017}. The author
constructed an infinite family of \emph{minimal} strongly chordal graphs that are not leaf powers (i.e., removing any vertex results in a leaf power). 

For fixed $k$, the conjecture that $\L(k)$ may be characterized by a finite set of obstructions remained open. 
Indeed, for $k\le 4$, the classes $\L(k)$ can be characterized as chordal graphs that do not contain any graph from $\F_k$ as induced subgraphs, where $\F_k$ is a finite set. Specifically:
\begin{itemize}
\item $k=2$: A graph is in $\L(2)$ {\em if and only if} it is a disjoint union of cliques. That is, $\L(2)$ is precisely the set of graphs that forbid $P_3$, the chordless path with three vertices, as an induced subgraph. Thus $|\F_2|=1$.
\item $k=3$: Dom et al.~\cite{DOM2004} gave the first characterization of $\L(3)$: a graph is in $\L(3)$ {\em if and only if} it is chordal and does not contain a bull, a dart or a gem as induced subgraph. Thus $|\F_3|=3$. Other characterizations of $\L(3)$ were later discovered~\cite{BRANDSTADT2006}
\item
$k=4$: Brandst\"adt, Bang Le and Sritharan~\cite{Brandstadt2008-4LP}
proved that a graph is in $\L(4)$ {\em if and only if} it is chordal and does not contain as induced subgraph one of a finite set $\F_4$ of graphs\footnote{Formally, they show that the set of basic 4-leaf power, where no two leaves of the leaf root share a parent, can be characterized by chordal graphs which do not have one of 8 graphs as induced subgraphs. $\F_4$ can be deduced from this set.}.
\end{itemize}
%In this paper, we focus on the characterization of $k$-leaf powers. 
Given this, the aforementioned conjecture naturally arose: for every $k$, is the class $\L(k)$ equivalent to the set of chordal graphs that do not contain as induced subgraphs any of a finite set $\F_k$ of graphs? 

For $k=5$, Brandst\"adt, Bang Le and Rautenbach~\cite{BRANDSTADT20093843} proved this is true for a special subclass of $\L(5)$. Specifically, the \emph{distance hereditary}\footnote{A graph $G$ is distance hereditary if for all pairs of vertices $(u,v)$ in all subgraphs of $G$ either the distance is the same as in $G$ or there is no path from $u$ to $v$.} $5$-leaf power graphs are chordal graphs that do contain a set of $34$ graphs as induced subgraphs.
However, for the general case, they state
\begin{quote}
    ``{\em For $k\geq 5$, no characterization of $k$-leaf powers is known despite considerable effort. Even the characterization of $5$-leaf powers appears to be a major open problem.''~\cite{BRANDSTADT20093843}} 
\end{quote}
The contribution of this paper is to disprove the conjecture: for all $k\geq 5$, it is impossible to characterize the set of $k$-leaf powers as the set of chordal graphs which are $\F_k$-free for $|\F_k|$ finite. In fact, we show that even for the more restrictive class of strongly chordal graphs it is impossible to characterize the set of $k$-leaf powers as the set of strongly chordal graphs which are $\F_k$-free for finite $|\F_k|$. 

Let us conclude this section by discussing the second major theme
in this area, namely, efficient recognition algorithms.
The computational complexity of deciding whether or not a graph is in $\L$ is wide open.  
We remark, however, that some graphs in $\L$ have a {\em leaf rank} that is exponential in the number of their vertices, where the leaf rank of a graph $G$ is the minimum $k$ such that $G \in \L(k)$~\cite{LowerBoundsLR}.  The question of computing the leaf rank of subclasses of $\L$ in polynomial time was recently initiated in~\cite{le2023computing}.

For fixed values of $k$, though, progress has been made in designing polynomial-time algorithms for the $\L(k)$ recognition problem.  For $\L(2), \L(3)$ and $\L(4)$, this immediately follows from the above characterizations because $\F_2, \F_3$ and $\F_4$ are finite. 
In fact, all these three recognition problems can be solved in linear time; see \cite{BRANDSTADT2006,Brandstadt2008-4LP}.
% Therefore, the $\L(2)$ recognition problem can be efficiently solved in linear time.
% Several other characterizations of $\L(3)$ were later discovered in \cite{BRANDSTADT2006} and used to design a linear-time algorithm for the $\L(3)$ recognition problem.
% From this characterization, the authors of \cite{Brandstadt2008-4LP} were able to derive a linear-time algorithm for the $\L(4)$ recognition problem.
Using a dynamic programming approach, Chang and Ko~\cite{Chang2007} described a linear-time algorithm for the $\L(5)$ recognition problem, and Ducoffe~\cite{ducoffe20194} proposed a polynomial-time algorithm for the $\L(6)$ recognition problem. Recently, Lafond~\cite{Lafond2023} designed a polynomial-time algorithm for the $\L(k)$ recognition problem, for any constant $k \geq 2$. 
The algorithm is theoretically efficient albeit completely impractical: the polynomial's exponent depends only on $k$ but is $\Omega(k \binuparrow\binuparrow k)$, that is, a tower of exponents $k^{k^{\rddots k}}$ of height $k$.  
We remark that the algorithm 
%essentially uses brute-force enumeration of possible trees that explain ``small'' subgraphs of $G$, and 
does not rely on specific characterizations of $k$-leaf power graphs aside from the fact that they are chordal. It appears difficult to significantly improve its running time without a better understanding of the graph theoretical structure of graphs in $\L(k)$.  Our work assists in this regard by improving our knowledge of $k$-leaf powers in terms of forbidden induced subgraphs.

% Note that if this conjecture holds true, it would offer an alternative proof for the existence of a polynomial-time algorithm for the $\L(k)$ recognition problem for any constant $k$. Moreover, if $\F_k$ and the graphs within $\F_k$ were reasonably small (for instance, at most exponential or even linear in $k$ as the known characterizations suggest), one could improve the running time of \cite{Lafond2023}. 

\subsection{Overview and Results}
We now present an overview of the paper and our results.
In Section~\ref{sec:proof} we present our main theorem:
\begin{restatable}{theorem}{main}\label{thm:main}
    For $k\geq 5$, the set of $k$-leaf powers cannot be characterized as the set of strongly chordal graphs which are $\F_k$-free, where $\F_k$ is a finite set of graphs.
\end{restatable}

There we discuss the three types of gadgets we need. These gadgets can be combined to form an infinite family of pairwise incomparable graphs which are not $k$-leaf powers.
We prove the main theorem modulo three critical lemmas on the gadgets. 
In Section~\ref{sec:gadgets} we present proofs of the three critical lemmas.
Finally, in Section~\ref{sec:linear} we show how to modify our proof
to derive a similar theorem for \emph{linear $k$-leaf powers}:
\begin{restatable}{theorem}{linear}\label{thm:linear}
    For $k\geq 5$, the set of linear $k$-leaf powers cannot be characterized as the set of strongly chordal graphs which are $\F_k$-free where $\F_k$ is a finite set of graphs.
\end{restatable}
Here, a linear $k$-leaf power is a graph that has a $k$-leaf root which is the subdivision of a \emph{caterpillar}. We remark that that the class of linear leaf powers can be recognised in linear time, as shown by Bergougnoux et al.~\cite{Bergougnoux}.

\section{The Proof Modulo Three Critical Lemmas}\label{sec:proof}

In this section we prove our main theorem, Theorem~\ref{thm:main},
assuming the validity of three critical lemmas.
The proofs of these lemmas form the main technical contribution of the paper and are deferred to Section~\ref{sec:gadgets}.

\subsection{Preliminaries}
Before presenting the proof of Theorem~\ref{thm:main}, we present  
necessary definitions and notations.
Let's start with a formal definition of $k$-leaf powers. Let $G = (V,E)$ be a simple finite graph, and $k\geq 2$ be an integer. $G$ is called a $k$-\emph{leaf power} if there exists a tree $T$, known as a {\em $k$-leaf root of $G$}, with the following properties:
\begin{itemize}
    \item $V$ is the set of leaves of $T$.
    \item For any pair of vertices $u,v \in V$, there is an edge $uv \in E$ {\em if and only if} the $d_T(u,v) \leq k$.
\end{itemize}
Here $d_T$ is the distance metric induced by the tree $T$ when two adjacent vertices are a distance of~$1$ apart.
To simplify the notation, we will use $d$ instead of $d_T$ when the context is clear. We will use the notation $\dist_G$ to denote distance within the graph $G$ and thus distinguish it from the distance $d_T$ induced by a leaf root $T$.

\subsection{The Proof of the Main Theorem}

To prove Theorem~\ref{thm:main}, for any $k\ge 5$, we will construct a collection of arbitrarily large strongly chordal graphs that are \emph{minimal non }$k$\emph{-leaf powers}. Specifically, these graphs have the property that any ``strict'' induced subgraph is a $k$-leaf power.

To accomplish this goal, we fix $k\ge 5$. We then begin by designing a graph $H_n$, for all $n\geq 0$,
built using three gadget graphs joined in series. First will be the {\em top gadget} and last the {\em bottom gadget}. In between will be exactly $n$ copies of the 
{\em interior gadget}.
We denote these gadget graphs by $\T, \B$ and $I$,
respectively. These gadget graphs will satisfy a set
of critical properties. To formalize these properties we require the following definition.
Given a graph $G=(V,E)$ and $T$ a $k$-leaf root of $G$. For $v\in V$, let $m_T(v)=\min_{u\in V\setminus \{v\}} d_T(u,v)$. That is, $m_T(v)$ is the shortest distance in the tree $T$ from the leaf $v$ to any other leaf $u$. 

The aforementioned properties of $\T, \B$ and $I$ are stated in the subsequent three critical lemmas.

\begin{restatable}{lemma}{top}
    \label{lem:top}
For all $k\geq 4 $, there exists a gadget graph $\T$ that contains a vertex $t\in V(\T)$ such that:
\begin{enumerate}
        \item For any $k$-leaf root $T$ of $\T$, $m_{T}(t) = 3$.
        \item There exists a $k$-leaf root $T_{\T}$ of  $\T$.
    \end{enumerate}
\end{restatable}

\begin{restatable}{lemma}{bot}\label{lem:bottom}
For all $k\geq 4$, there exists a gadget graph $\B$ that contains a vertex $b \in V(\B)$ such that:
\begin{enumerate}
        \item For any $k$-leaf root $T$ of $\B$, $m_T(b) \leq k-1$.
        \item There exists a $k$-leaf root $T_{\B}$ such that $m_{T_{\B}}(b) = k-1$
\end{enumerate}
\end{restatable}   
\begin{restatable}{lemma}{ind}\label{lem:induction}
For all $k\geq 5$, there exists a gadget graph $I$ that contains two distinct vertices $t_I,b_I\in V(I)$ such that:
\begin{enumerate}
        \item For all $k$-leaf roots $T$ of $I$, $m_T(t_I)\geq k\Longrightarrow m_T(b_I)=3$.
        \item There exists a $k$-leaf root $T_I$ of $I$ such that $m_{T_I}(t_I)=k$ and $m_{T_I}(b_I)=3$.
        \item There exists a $k$-leaf root $R_I$ of $I$ such that $m_{R_I}(t_I)=k-1$ and $m_{R_I}(b_I)=4$.
    \end{enumerate}    
\end{restatable}

We will prove the existence of gadget graphs
$\T, \B$ and $I$ required to verify the three lemmas in Section~\ref{sec:gadgets}.
For the rest of the section, we will assume these lemmas and use them to prove our main result.

\begin{figure}[h!]
    \centering
    \includegraphics[height=10cm]{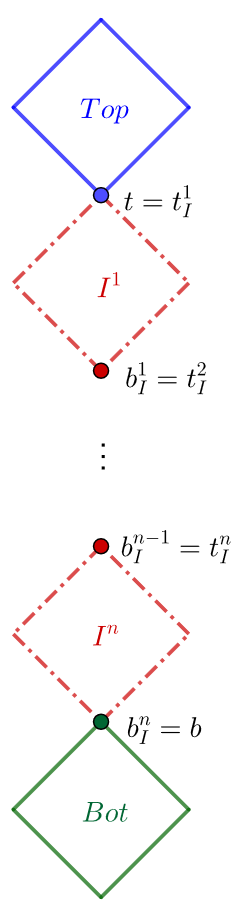}
    \caption{The construction of $H_n$}\label{fig:Hn}
\end{figure}

First, as alluded to above, we then combine our three gadgets to create an intermediary graph $H_n$. In particular, $H_n$ is the graph obtained by connecting in series one copy of $\T$,
then $n$ copies of $I$: $I^1, \dots, I^n$ and finally one copy of $\B$.  
This construction is illustrated in \Cref{fig:Hn}.
The vertices $t_I$ and $b_I$, mentioned in Lemma~\ref{lem:induction}, of the $j$-th copy $I^j$ are denoted $t^j_I$ and $b^j_I$, respectively.
Notice that to connect the gadgets within $H_n$, we identify the vertices described in Lemmas~\ref{lem:top}, \ref{lem:bottom}, and \ref{lem:induction} as follows. We identify $t$ with $t^1_I$, for all $j < n$, $b^j_I$ with $t^{j+1}_I$, and finally, $b^n_I$ with $b$. As a special case when $n = 0$, the graph $H_0$ is obtained by taking $\T$ and $\B$ and identifying $t$ with $b$.

In order to prove Theorem~\ref{thm:main} we must study the structure of $H_n$. We denote by $H_n - \T$ (resp. $H_n - \B$) the graph obtained from $H_n$ by deleting the top gadget $\T$ (resp. the bottom gadget $\B$), i.e. removing all vertices of $\T$ (resp. $\B$) except for the common vertex $t = t_I^1$ (resp. $b = b_I^n$). Of importance is the next lemma.
\begin{lemma}\label{lem:Hn}
    The graph $H_n$ has the following properties:
    \begin{enumerate}
        \item $\dist_{H_n}(b,t) \geq n$. \label{cond:increasing}
        \item $H_n$ is strongly chordal. \label{cond:Str-Chordal}
        \item $H_n - \T$ and $H_n - \B$ are both $k$-leaf powers. \label{cond:Top/Bot}
        \item $H_n$ is not a $k$-leaf power. \label{cond:not-LP}
        \end{enumerate}
\end{lemma}
\begin{proof}
In order to prove \ref{cond:increasing}, the distance between $t_I$ and $b_I$ in $I$ is at least $1$ because the two vertices are distinct. Hence $\dist_{H_n}(b,t) \geq n$, because there are $n$ copies of the interior gadget $I$.

For the proof of \ref{cond:Str-Chordal}, the three gadgets are $k$-leaf powers and, therefore, are strongly chordal. The construction of $H_n$ does not introduce additional cycles; thus, $H_n$ remains strongly chordal.

To prove \ref{cond:Top/Bot}, we combine the leaf-root properties provided by the gadgets $T_{\T}$, $T_{\B}$, $T_I$, and $R_I$, as illustrated in Figure~\ref{fig:Trees}.
\begin{figure}[h!]
    \centering
    \includegraphics[height=12cm]{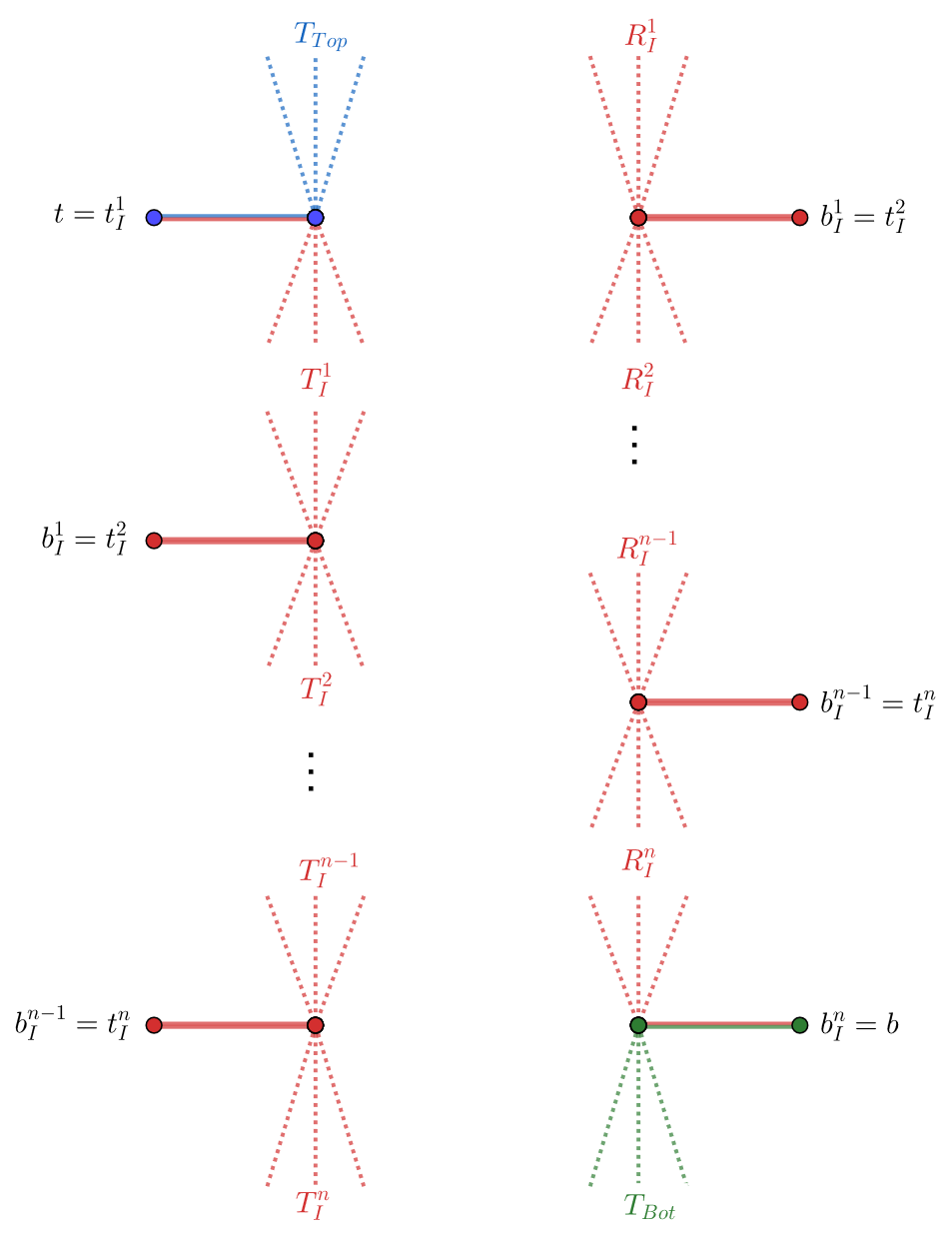}
    \caption{The $k$-leaf roots of  $H_n - \B$ and $H_n - \T$}\label{fig:Trees}
\end{figure}

The left tree in Figure~\ref{fig:Trees} is obtained by merging $T_{\T}$ with $n$ copies of $T_I$, denoted as $T_I^1, \dots, T_I^n$. We identify $t$ with $t_I^1$, and we identify the
of $t$ in $T_{\T}$ with $t_I^1$ and with the parent of $t_I^1$ in $T_I^1$. Similarly, for all $j \leq n-1$, we identify $b_I^j$ and its parent with $t_I^{j+1}$ and its parent, respectively. We now prove that the resulting tree is a $k$-leaf root of $H_n - \T$.
$\T$ and each copy of the interior gadget $I$ are the $k$-leaf power of the corresponding subtree: $T_{\T}$ for $\T$ and $T_I^j$ for the $j$-th copy of $I$. It remains to show that we do not introduce any additional, unwanted edges. If $v_1$ is a leaf of $T_{\T}$ different from $t$, and $v_2$ is a leaf of $T_I^1$ different from $t^1_I$ then, using Lemma~\ref{lem:top}, we conclude $d(v_1,t) \geq m_{T_{\T}}(t) = 3$. Furthermore, using the second point of Lemma~\ref{lem:induction}, we conclude $d(v_2,t_I^1) \geq m_{T_I}(t_I) = k$. Therefore, $d(v_1,v_2) \geq d(v_1,t) + d(v_2,t_I^1)-2 \geq k+3-2 = k+1 > k$, and there is no edge between $v_1$ and $v_2$ in the $k$-th power of the tree. Similarly, if $v_1$ is a leaf of $T_I^j$ different from $b_I^j$ for some $j \leq n-1$ and $v_2$ is a leaf of $T_I^{j+1}$ different from $t_I^{j+1}$, we have $d(v_1,b_I^j) \geq 3$ and $d(v_2, t_I^{j+1}) \geq k$. Thus, $d(v_1,v_2) \geq d(v_1,b_I^j) + d(v_2, t_I^{j+1}) -2 \geq  k+3-2 = k+1 > k$, and there is no edge between $v_1$ and $v_2$ in the $k$-th power of the tree.

The right tree in Figure~\ref{fig:Trees} is formed by merging $n$ copies of $R_I$, denoted $R^1_I, \dots, R^n_I$, with $T_{\B}$. For all $j \leq n-1$, we identify $b_I^j$ with $t_I^{j+1}$,
and we identify the parent of $b_I^j$ with the parent of $t_I^{j+1}$. Finally, we identify $b_I^n$ and its parent in $R^n_I$ with $b$ and its parent in $T_{\B}$, respectively. Similar to the left tree, we must prove that no additional, unwanted edges are created. If $v_1$ is a leaf of $R^j_I$ different from $b_I^j$ for some $j \leq n-1$ and $v_2$ is a leaf of $R^{j+1}_I$ different from $t_I^{j+1}$ then, using Lemma~\ref{lem:induction}, we conclude $d(v_1,b_I^j) \geq m_{R_I}(b_I) = 4$ and $d(v_2, t_I^{j+1}) \geq m_{R_I}(t_I) = k-1$. Thus, $d(v_1,v_2) \geq d(v_1,b_I^j) + d(v_2, t_I^{j+1}) -2 \geq k-1 + 4-2 = k+1 > k$, and there is no edge between $v_1$ and $v_2$ in the $k$-th power of the tree. Similarly if $v_1$ is a leaf of $R^n_I$ different from $b_I^n$ and $v_2$ is a leaf of $T_{\B}$ different from $b$ then we have $d(v_1,b_I^n) \geq 4$ and, using Lemma~\ref{lem:bottom},  $d(v_2,b)  \geq k-1$. Therefore $d(v_1,v_2) \geq d(v_1,b_I^n) + d(v_2,b) -2 \geq  k+1>k$ and $v_1$,$v_2$ are not connected in the $k$-th power of the tree.
This completes the proof of \ref{cond:Top/Bot}, the third point of the lemma.

It remains to prove \ref{cond:not-LP}, the final point of the lemma. We start by proving by induction that for any integer $n$, in any $k$-leaf root $T$ of $H_n - \B$, there is a leaf at distance $3$ of $b$ in $T$. When $n = 0$, there are no gadgets $I$ between $\B$ and $\T$, so $b = t$ and the property holds by property 1 of Lemma~\ref{lem:top}. Turning to the induction step, assume that the property holds for some $n \geq 0$ and consider a $k$-leaf root $T$ of $H_{n+1}- \B$. Since $H_n - \B$ is an induced subgraph of $H_{n+1} - \B$, some induced subgraph of $T$ is a $k$-leaf root of $H_n- \B$. By the induction hypothesis, there exists a vertex $v_1$ in $H_n- \B$ at a distance exactly $3$ from $b_I^n = t_I^{n+1}$ in an induced subgraph of $T$. Adding vertices will not alter the distance, so $d(v_1,b_I^n) = 3$ in $T$.  We claim that every vertex of $I^{n+1}$ is at distance at least $k$ from $b_I^n = t_I^{n+1}$ in $T$ (except $b_I^n$ itself). Assume, by contradiction, that there exists a vertex $v_2$  in the last copy $I^{n+1}$, distinct from $b_I^n$ such that that $d(v_2, b_I^n) \leq k-1$ . This assumption would imply that $d(v_1, v_2) \leq d(v_1, b_I^n) + d(v_2, b_I^n) - 2 \leq 3 + (k-1) - 2 = k$, meaning that $v_1$ and $v_2$ are connected in the $k$-th power of $T$, contradicting the fact that $T$ is a $k$-leaf root of $H_{n+1}-\B$. Therefore, all vertices in $I^{n+1}$, distinct from $b_I^n$, are at a distance of at least $k$ from $b_I^n$ in $T$. We can now apply property 1 of Lemma~\ref{lem:induction}, which concludes the induction.

Now assume by contradiction that there exists a $k$-leaf root $T$ of $H_n$. $T$ induces a $k$-leaf root of $H_n - \B$. In particular, there exists a vertex $v_1$ in $H_n- \B$ such that $d(b_I^n, v_1) = 3$. Moreover, property 1 of Lemma~\ref{lem:bottom} implies that there exists a vertex $v_2$ in $\B$, distinct from $b$, such that $d(v_2, b) \leq k-1$. Combining these equations, we get $d(v_1, v_2) \leq k-1 + 3 - 2 = k$, contradicting the fact that there is no edge between $v_1$ and $v_2$. This proves that $H_n$ is not a $k$-leaf power, as desired.
\end{proof}

As stated $H_n$ is an intermediate graph in proving the main result. 
We will actually show the existence
of an induced subgraph $G_{k,n}$ of $H_n$
that is strongly chordal and minimal non $k$-leaf power.
More precisely, we have the following lemma.

\begin{lemma}\label{lem:Gn}
    For all $k\geq 5$ and $n\geq 0$, there exists a graph $G_{k,n}$ such that:
\begin{enumerate}
    \item $G_{k,n}$ is strongly chordal and contains at least $n$ vertices.
    \item $G_{k,n}$ is not a $k$-leaf power.
    \item If $G\neq G_{k,n}$ is an induced subgraph of $G_{k,n}$ then $G$ is a $k$-leaf power.
\end{enumerate}
\end{lemma}
\begin{proof}
Let $G_{k,n}$ be a minimal induced subgraph of $H_n$ that is not a $k$-leaf power. It is strongly chordal because, by Lemma~\ref{lem:Hn}, $H_n$ is strongly chordal. By definition, $G_{k,n}$ is not in $\mathcal{L}(k)$, but every induced subgraph $G$ of $G_{k,n}$ is in $\mathcal{L}(k)$ if $G\neq G_{k,n}$. It remains to prove that $|G_{k,n}| \geq n$. By Lemma~\ref{lem:Hn}, both $H_n - \B$ and $H_n - \T$ are in $\mathcal{L}(k)$. Therefore, $G_{k,n}$ must contain a vertex from $\T$ and a vertex from $\B$. Moreover, $G_{k,n}$ is connected by its minimality, since the disjoint union of $k$-leaf powers is a $k$-leaf power. Hence, $G_{k,n}$ must contain a path from $\T$ to $\B$. By Lemma~\ref{lem:Hn}, $\dist_{H_n}(b,t) \geq n$, and therefore $|G_{k,n}| \geq n$.
\end{proof}

Our main result follows directly from Lemma~\ref{lem:Gn}

\begin{proof}[Proof of \Cref{thm:main}]
If $\L(k)$ is the set of strongly chordal graphs which are $\F_k$-free, then $\F_k$ must contain $G_{k,n}$ for all $n$ because it is strongly chordal and a minimal non $k$-leaf power by Lemma~\ref{lem:Gn}. But the set $\{G_{k,n}, n \geq 0\}$ is infinite because $G_{k,n}$ has more than $n$ vertices. Therefore $\F_k$ must be infinite, which concludes the proof of \Cref{thm:main}.
\end{proof}

\section{The Gadget Graphs}\label{sec:gadgets}

So we have proven the main theorem modulo the three critical lemmas.
Recall to prove these lemmas we must construct the appropriate three gadget graphs, namely $\T$, $\B$ and $I$.
We present these constructions and give formal proofs of 
Lemmas~\ref{lem:top},~\ref{lem:bottom} and~\ref{lem:induction} in this section.

We start with a general observation. In a tree $T$, if a pair of leaves are a distance of~$2$ apart, they share the same parent. Consequently, their distances to every other leaf are identical. A consequence of this is that if two vertices are not connected by an edge, or if they have different neighborhoods in a graph, they must be at a distance of at least $3$ in any leaf root of that graph. In the gadgets we describe in this section, any two vertices connected by an edge always have distinct neighborhoods. Therefore, we assume that for any pair of vertices $x$ and $y$ and any leaf root $T$, we have $d_T(x, y) \geq 3$.

\subsection{The Top Gadget}\label{sec:top}

We begin by showing the existence of an appropriate top gadget, $\T$.

\top*
\begin{proof}
Let $P$ be the path on the $2k-3$ vertices $v_1, \dots, v_{2k-3}$. The top gadget $\T$ is defined as $P^{k-2}$ with $t = v_{k-2}$.
Two examples of $\T$ are shown in Figure~\ref{fig:top-gadget}, for the cases $k\in \{5,6\}$.
First, note that $\T$ is a $k$-leaf power. In particular, a $k$-leaf root of $\T$ is the caterpillar $T_{\T}$, which is constructed by attaching a leaf to every vertex of $P$. For the first property, we use Lemma $2$ from \cite{WAGNER20095505} which implies that in any $k$-leaf root $T$ of $\T$, we have $d(v_{k-3}, v_{k-2}) \leq 3$.
\end{proof}

\begin{figure}[!h]
    \centering
    \begin{subfigure}{0.45\textwidth}
        \centering
        \begin{tikzpicture}
            %P_9^3
            \node (P1) at (225:3) {$v_1$};
            \node (P2) at (180:3) {$v_2$};
            \node (P3) at (135:3) {$\hspace{-4mm}t=v_3$};
            \node (P4) at (90:3) {$v_4$};
            \node (P5) at (45:3) {$v_5$};
            \node (P6) at (0:3) {$v_6$};
            \node (P7) at (315:3) {$v_7$};
            \draw (P1) -- (P2) -- (P3) -- (P4) -- (P5) -- (P6) -- (P7);
            \draw (P1) -- (P3) -- (P5) -- (P7);
            \draw (P2) -- (P4) -- (P6);
            \draw (P1) -- (P4) -- (P7);
            \draw (P2) -- (P5);
            \draw (P3) -- (P6);
        \end{tikzpicture}
        %\subcaption{Top Gadget for $k=5$.}
    \end{subfigure}
    \quad
    \begin{subfigure}{0.45\textwidth}
        \centering
        \begin{tikzpicture}
            %P_9^3
            \node (P1) at (250:3) {$v_1$};
            \node (P2) at (210:3) {$v_2$};
            \node (P3) at (170:3) {$v_3$};
            \node (P4) at (130:3) {$\hspace{-4mm}t=v_4$};
            \node (P5) at (90:3) {$v_5$};
            \node (P6) at (50:3) {$v_6$};
            \node (P7) at (10:3) {$v_7$};
            \node (P8) at (330:3) {$v_8$};
            \node (P9) at (290:3) {$v_9$};
            % \node (P1) at (226:3) {$v_1$};
            % \node (P2) at (192:3) {$v_2$};
            % \node (P3) at (158:3) {$v_3$};
            % \node (P4) at (126:3) {$v_4$};
            % \node (P5) at (90:3) {$v_5$};
            % \node (P6) at (56:3) {$v_6$};
            % \node (P7) at (22:3) {$v_7$};
            % \node (P8) at (348:3) {$v_8$};
            % \node (P9) at (314:3) {$v_9$};
            \draw (P1) -- (P2) -- (P3) -- (P4) -- (P5) -- (P6) -- (P7) -- (P8) -- (P9);
            \draw (P1) -- (P3) -- (P5) -- (P7) -- (P9);
            \draw (P2) -- (P4) -- (P6) -- (P8);
            \draw (P1) -- (P4) -- (P7);
            \draw (P2) -- (P5) -- (P8);
            \draw (P3) -- (P6) -- (P9);
            \draw (P1) -- (P5);
            \draw (P2) -- (P6);
            \draw (P3) -- (P7);
            \draw (P4) -- (P8);
            \draw (P5) -- (P9);
        \end{tikzpicture}
        %\subcaption{Top Gadget for $k=6$.}
    \end{subfigure}
    % \begin{subfigure}{0.45\textwidth}
    %     \centering
    %     \begin{tikzpicture}
    %         %P_11^4
    %         \node (P1) at (240:3) {$v_1$};
    %         \node (P2) at (210:3) {$v_2$};
    %         \node (P3) at (180:3) {$v_3$};
    %         \node (P4) at (150:3) {$v_4$};
    %         \node (P5) at (120:3) {$\hspace{-4mm}t=v_5$};
    %         \node (P6) at (90:3) {$v_6$};
    %         \node (P7) at (60:3) {$v_7$};
    %         \node (P8) at (30:3) {$v_8$};
    %         \node (P9) at (360:3) {$v_9$};
    %         \node (P10) at (330:3) {$v_{10}$};
    %         \node (P11) at (300:3) {$v_{11}$};
    %         \draw (P1) -- (P2) -- (P3) -- (P4) -- (P5) -- (P6) -- (P7) -- (P8) -- (P9) -- (P10) -- (P11);
    %         \draw (P1) -- (P3) -- (P5) -- (P7) -- (P9) -- (P11);
    %         \draw (P2) -- (P4) -- (P6) -- (P8) -- (P10);
    %         \draw (P1) -- (P4) -- (P7) -- (P10);
    %         \draw (P2) -- (P5) -- (P8) -- (P11);
    %         \draw (P3) -- (P6) -- (P9);
    %         \draw (P1) -- (P5) -- (P9);
    %         \draw (P2) -- (P6) -- (P10);
    %         \draw (P3) -- (P7) -- (P11);
    %         \draw (P4) -- (P8);
    %         \draw (P1) -- (P6);
    %         \draw (P2) -- (P7);
    %         \draw (P3) -- (P8);
    %         \draw (P4) -- (P9);
    %         \draw (P5) -- (P10);
    %         \draw (P6) -- (P11);
    %     \end{tikzpicture}
    %     \subcaption{Top Gadget for $k=7$.}
    % \end{subfigure}
    \caption{The Top Gadget for $k=5$ and $k=6$.}
    \label{fig:top-gadget}
\end{figure}
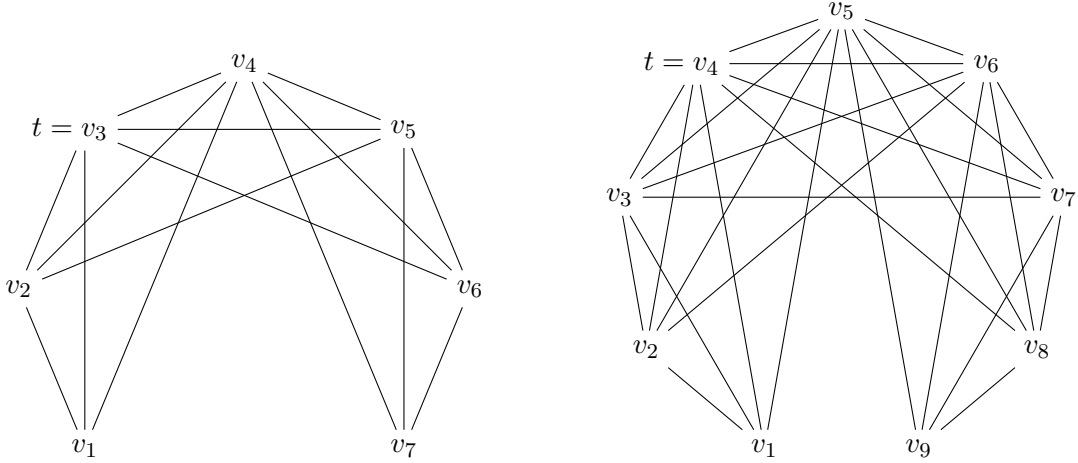

\subsection{The Bottom Gadget}\label{sec:bot}

Next, we construct the bottom gadget, $\B$. 
A key technical tool we require is the {\em 4-Point Condition}. 
This is the following classical characterization of tree metrics. 
\begin{theorem}[4-Point Condition]\label{thm:4-PC} \cite{BUNEMAN197448}
    Let $d$ be a distance on a finite set $V$, then there exists a tree $T$ whose leaves are $V$ such that $\forall u,v\in V$ $d_T(u,v)=d(u,v)$ if and only if the following condition is true for all $(u,v,w,t)\in V$:
    \[d(u,v)+d(w,t)\leq \max\set{d(u,w)+d(v,t),d(v,w)+d(u,t)}.\]
\end{theorem}

Our bottom gadget $\B$ will simply be a {\em diamond}, the complete graph on 4 vertices minus one edge.
Consequently, we begin by proving the following corollary of the $4$-Point Condition when applied to a diamond.

\begin{corollary}\label{cor:diamond}
    In any $k$-leaf root $T$ of a diamond with vertex set $\set{b,v_1,v_2,v_3}$ where $(v_1,v_3)\notin E$, $d(b,v_2)\neq k$.
\end{corollary}

\begin{proof}
    Assume for contradiction that $d(b, v_2) = k$.
    Then since $d(v_1,v_3)>k$, we get $d(v_1,v_3)+d(b,v_2)>2k$. 
    
    On the other hand, since we have a leaf root of the diamond, we must have:
    \[\max\set{d(v_1,v_2),d(v_2,v_3),d(v_3,b),d(b,v_1)}\leq k.\]
    
    This implies that $d(v_1,v_2)+d(v_3,b)\leq 2k$ and $d(v_2,v_3)+d(b,v_1)\leq 2k$. 

    So, $d(v_1,v_3)+d(b,v_2)>\max\set{d(v_1,v_2)+d(v_3,b),d(v_2,v_3)+d(b,v_1)}$ which contradicts Theorem~\ref{thm:4-PC}
\end{proof}

Corollary~\ref{cor:diamond} allows us to prove our critical lemma for the bottom gadget.

\bot*

\begin{proof}
Let the graph $\B$ be the diamond with vertex set $\{b, v_1, v_2, v_3\}$ and non-edge $(v_1, v_3)$. By Corollary~\ref{cor:diamond}, $d(b,v_2)\leq k-1$ in any leaf root. Thus the first property holds.  For the second property, there are two cases, illustrated in Figure~\ref{fig:diam-leaf-root}, depending upon the parity of $k$. 

    \begin{figure}[h!]
        \centering
        \begin{subfigure}{0.45\textwidth}
            \centering
            \begin{tikzpicture}
                \node (b) at (-3,0) {$b$};
                \node (v1) at (0,3) {$v_1$};
                \node (v2) at (3,0) {$v_2$};
                \node (v3) at (0,-3) {$v_3$};
                \node[label=above:{$\frac{k-1}{2}$}] (eb) at (-1.5,0) { };
                \node[label=right:{$\frac{k+1}{2}$}] (ev1) at (0,1.5) { };
                \node[label=above:{$\frac{k-1}{2}$}] (ev2) at (1.5,0) { };
                \node[label=right:{$\frac{k+1}{2}$}] (ev3) at (0,-1.5) { };
                \node (O) at (0,0) {O};
                \draw[ultra thick] (b) -- (O); 
                \draw[ultra thick] (v1) -- (O); 
                \draw[ultra thick] (v2) -- (O); 
                \draw[ultra thick] (v3) -- (O); 
                \draw[dotted] (b) -- (v1) -- (v2) -- (v3) -- (b);
                \draw[dotted]  (b) to [out=-15,in=-165] (v2);
            \end{tikzpicture}
            \subcaption{$T_{\B}$ for $k$ odd}
        \end{subfigure}
        \begin{subfigure}{0.45\textwidth}
            \centering
            \begin{tikzpicture}
                \node (b) at (-3,0) {$b$};
                \node (v1) at (-0.5,3) {$v_1$};
                \node (v2) at (3,0) {$v_2$};
                \node (v3) at (0.5,-3) {$v_3$};
                \node[label=above:{$\frac{k}{2}-1$}] (eb) at (-1.5,0) { };
                \node[label=right:{$\frac{k}{2}$}] (ev1) at (-0.5,1.5) { };
                \node[label=above:{$\frac{k}{2}-1$}] (ev2) at (1.5,0) { };
                \node[label=left:{$\frac{k}{2}$}] (ev3) at (0.5,-1.5) { };
                \node (O1) at (-0.5,0) {$O_1$};
                \node (O2) at (0.5,0) {$O_2$};
                \draw (O1) -- (O2);
                \draw[ultra thick] (b) -- (O1); %k/2-1
                \draw[ultra thick] (v1) -- (O1); %k/2
                \draw[ultra thick] (v2) -- (O2); %k/2-1
                \draw[ultra thick] (v3) -- (O2); %k/2
                \draw[dotted] (b) -- (v1) -- (v2) -- (v3) -- (b);
                \draw[dotted]  (b) to [out=-15,in=-165] (v2);
            \end{tikzpicture}
            \subcaption{$T_{\B}$ for $k$ even}
        \end{subfigure}
        \caption{The $k$-leaf roots of the diamond with $\min_{v\in V(D)\setminus \{b\}} d(b,v)=k-1$. Here the bold edges denote paths of the described length; the dotted edges are the edges of the diamond.}
        \label{fig:diam-leaf-root}
    \end{figure}
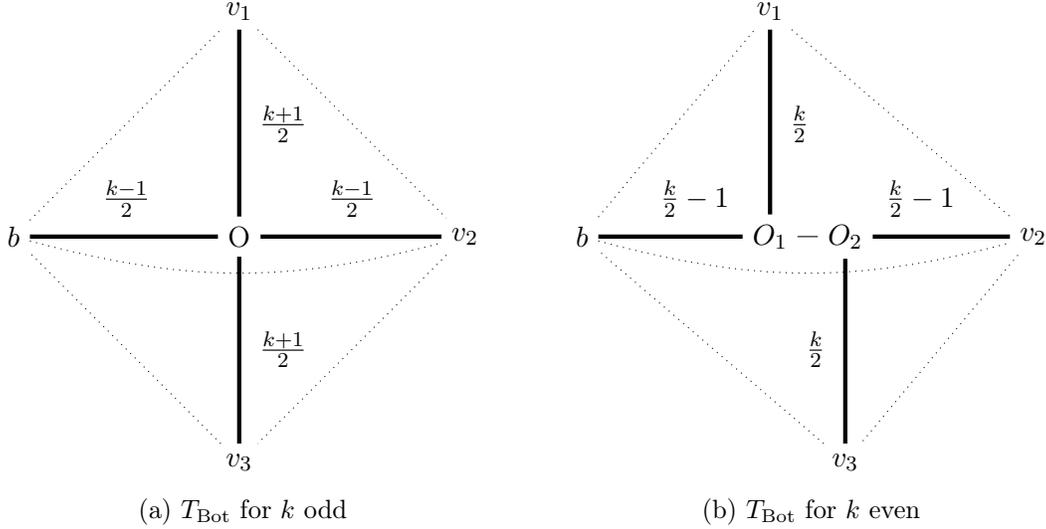
    
    \begin{itemize}
        \item If $k$ is odd, start with $b$ and $v_2$ at distance $k-1$. Let $O$ be the midpoint of the two at distance $\frac{k-1}{2}$ from both. Set $v_1$ and $v_3$ to be each at distance $\frac{k+1}{2}$ from $O$. Then $b$ and $v_2$ will both be at distance exactly $k$ from both $v_1$ and $v_3$, but $v_1$ and $v_3$ are at distance $k+1$ from each other. Thus, the only distance which is greater than $k$ is $d(v_1,v_3)$ and the closest vertex to $b$ is $v_2$, as desired. 
        \item If $k$ is even, start with $b$ and $v_2$ at distance $k-1$. Set $O_1$ to be the point at distance $\frac{k}{2}-1$ from $b$ and $O_2$ the point at distance $\frac{k}{2}-1$ from $v_2$. Add $v_1$ at distance $\frac{k}{2}$ from $O_1$ and $v_3$ at distance $\frac{k}{2}$ from $O_2$. Then $b$ is at distance $k-1$ from $v_1$ and $k$ from $v_3$, while $v_2$ is at distance $k$ from $v_1$ and $k-1$ from $v_3$. Note that $v_1$ and $v_3$ are at distance $k+1$ from each other. Thus, the only distance which is greater than $k$ is $d(v_1,v_3)$ and the closest vertex to $b$ is $v_2$, as desired. 
    \end{itemize}
This lemma follows.
\end{proof}

\subsection{The Interior Gadget}\label{sec:int}

Lastly, we have the most complex construction, that of 
the interior gadget, $I$. 
Now we require the following lemma which, again, is a consequence of the 4-Point Condition.

\begin{lemma}\label{lem:closer}
    If $d(t,x_1)\leq \min\set{d(t,x_2),d(t,x_3)}$ and $d(y,x_1)>\max\set{d(y,x_2),d(y,x_3)}$, then:
    \[d(t,x_1)+d(x_2,x_3)<d(t,x_2)+d(x_1,x_3)=d(t,x_3)+d(x_1,x_2)\]
\end{lemma}

\begin{proof}
    Assume that $d(t,x_1)+d(x_2,x_3)\geq \max\set{d(t,x_2)+d(x_1,x_3),d(t,x_3)+d(x_1,x_2)}$. Then, by using this assumption and our bound on $d(t,x_1)$, we get:
    \begin{align*}
        d(x_2,x_3)&\geq \max\set{d(x_1,x_3)+(d(t,x_2)-d(t,x_1)),d(x_1,x_2)+(d(t,x_3)-d(t,x_1))}\\
        &\geq \max\set{d(x_1,x_3),d(x_1,x_2)}
    \end{align*}

    Then, by combining this bound with our bound on $d(y,x_1)$, we get:

    \[d(y,x_1)+d(x_2,x_3)>\max\set{d(y,x_2)+d(x_1,x_3),d(y,x_3)+d(x_1,x_2)}\]

    This contradicts Theorem~\ref{thm:4-PC}. This implies that the assumption is wrong, that is, we must have:

    \[d(t,x_1)+d(x_2,x_3)<\max\set{d(t,x_2)+d(x_1,x_3),d(t,x_3)+d(x_1,x_2)}.\]
    
    Now, assume without loss of generality that $d(t,x_2)+d(x_1,x_3)\geq d(t,x_3)+d(x_1,x_2)$. Then, by Theorem~\ref{thm:4-PC}, we must have $d(t,x_2)+d(x_1,x_3)\leq \max\set{d(t,x_1)+d(x_2,x_3),d(t,x_3)+d(x_1,x_2)}$. Since $d(t,x_2)+d(x_1,x_3)>d(t,x_1)+d(x_2,x_3)$, this implies that $d(t,x_2)+d(x_1,x_3)\leq d(t,x_3)+d(x_1,x_2)$. So we get that $d(t,x_2)+d(x_1,x_3)$ is bounded above and below by $d(t,x_3)+d(x_1,x_2)$ so they must be equal. 
    
    So we get:
    \[d(t,x_1)+d(x_2,x_3)<d(t,x_2)+d(x_1,x_3)=d(t,x_3)+d(x_1,x_2).\]
    This completes the proof.
\end{proof}

We will also use the following simple lemma:

\begin{lemma}\label{lem:even}
    For any 3 leaves $u,v,w$ of a tree, $d(u,v)+d(u,w)+d(v,w)$ is even.
\end{lemma}

\begin{proof}
    Since we have a tree, there is a unique vertex $O$ which is simultaneously in the path from $u$ to $v$, the path from $v$ to $w$ and the path from $u$ to $w$. 
    
   Hence $d(u,v)+d(u,w)+d(v,w)=2\cdot (d(u,O)+d(v,O)+d(w,O))$ which must be even. 
\end{proof}

We now have all the tools needed to prove our critical lemma for the interior gadget.

\ind*

Before proving this lemma, let's discuss the requirement that $k\ge 5$.
First observe that no such graph can exist for $k\leq 2$ because if $m_T(b_I)=3$ then $b_I$ is an isolated vertex in $I$. Thus its distance to other leaves does not matter as long as it's large enough, so the lemma could not hold.
Similarly, if $k=3$, the existence of a $3$-leaf root $T_I$ implies that $b_I$ is not an isolated vertex in $I$. But the existence of a $3$-leaf root $R_I$ implies that $b_I$ is an isolated vertex, a contradiction. 
Finally, for $k=4$, while there is no direct simple proof that the statement does not hold for any graph, 
the existence of a characterization of $4$-leaf powers implies that no such graph can exist.

\begin{proof}
    For convenience, we denote $t_I$ and $b_I$ by $t$ and $b$, respectively.
    To prove the lemma, we will explicitly construct $I$ for any value of $k\geq 5$. In order to do so, we consider two cases
    depending, again, upon the parity of $k$.

    For $k$ odd, set $q=\frac{k-1}{2}$. In particular, for $k\geq 5$ we must have $q\geq 2$. 
    
    We construct the graph $I$ using the following sets of vertices: 
    \begin{itemize}
        \item $t$ and $b$
        \item $X=\set{x_1,\dots, x_q}$
        \item $Y=\set{y_2,\dots, y_q}$
    \end{itemize}

    The edge set is defined as follows:
    \begin{itemize}
        \item For $i=1,\dots q$, $(t,x_i)$ and $(b,x_i)$ are edges. That is, $t$ and $b$ are adjacent to all vertices in $X$.
        \item For $i=1,\dots, q$, for $j=i+1,\dots q$, $(x_i,x_j)$ is an edge. That is, $X$ forms a clique. 
        \item For $i=2\dots q$, for $j=i\dots q$, $(y_i,x_j)$ is an edge. 
        \item $(b,y_q)$ is an edge.
    \end{itemize}
Equivalently, it will be helpful to define the set of edges using the neighborhood of each vertex:
    \begin{itemize}
        \item $t$ is adjacent to $X=\set{x_1,\dots,x_q}$.
        \item $b$ is adjacent to $X$ and to $y_q$.
        \item For $i=1,\dots, q$, $x_i$ is adjacent to $t$, to $b$, to $X\setminus \{x_i\}$ and to $y_j$ for $j=2,\dots, i$ (with $x_1$ having no neighbor in $Y$). 
        \item For $i=2,\dots, q$, for $j=i,\dots, q$, $y_i$ is adjacent to $x_j$. If $i=q$, then $y_q$ is also adjacent to $b$.
    \end{itemize}
    
    That is, we take $I=(V,E)$ to be defined by:
    \begin{equation}
        \begin{aligned}
            V=&\set{t,b}\cup\left(\bigcup_{i=1}^q \set{x_i}\right)\cup\left(\bigcup_{i=2}^{q} \set{y_i}\right)\\
            E=&\left(\bigcup_{i=1}^q \set{(t,x_i),(b,x_i)}\right)\cup \left(\bigcup_{1\leq i<j\leq q}\set{(x_i,x_j)}\right)\cup\left(\bigcup_{2\leq i\leq j\leq q}\set{(y_i,x_j)}\right)\cup\set{(b,y_q)}
        \end{aligned}
        \label{eq:odd-gadget}
    \end{equation}
This construction is illustrated in Figure~\ref{fig:interior-odd}.
We remark that this construction only makes sense for $k\geq 5$ because if $k=3$ or $k=1$ then $Y$ is not well defined.

        \begin{figure}[!ht]
        \centering
            \begin{tikzpicture}
                \node[circle] (T) at (0,4.5) {$t$};
                \node[circle] (B) at (0,-4.5) {$b$};
                \node[circle] (X1) at (-5,3) {$x_1$};
                \node[circle] (X2) at (-5,2) {$x_2$};
                \node[circle] (X3) at (-5,1) {$x_3$};
                \node[circle] (Xdots) at (-5,0.5) {$\vdots$};
                \node[circle] (XQ-2) at (-5,-1) {$x_{q-2}$};
                \node[circle] (XQ-1) at (-5,-2) {$x_{q-1}$};
                \node[circle] (XQ) at (-5,-3) {$x_q$};
                
                \node[above right] (XQ-1-label) at (-4.5,-2) {};
                \node[above right] (XQ-2-label) at (-4,-1){};
                \node[above right] (X3-label) at (-3.5,1) {};
                \node[above right] (X2-label) at (-3,2) {};
                \node[above right,draw] (X-label) at (-4,3.5) {$X$ (Clique)};
                
                \node [align=left,fit=(XQ-1)(XQ)(XQ-1-label), draw] (XQ-1-XQ){};
                \node [align=left,fit=(XQ-2)(XQ-2-label)(XQ-1-XQ), draw] (XQ-2-XQ) {};
                \node [align=left,fit=(X3)(XQ-2-XQ)(X3-label), draw] (X3-XQ){};
                \node [align=left,fit=(X2)(X3-XQ)(X2-label), draw] (X2-XQ){};

                \node[circle] (Y2) at (4,2) {$y_2$};
                \node[circle] (Y3) at (4,1) {$y_3$};
                \node[circle] (Ydots) at (4,0) {$\vdots$};
                \node[circle] (YQ-2) at (4,-1) {$y_{q-2}$};
                \node[circle] (YQ-1) at (4,-2) {$y_{q-1}$};
                \node[circle] (YQ) at (4,-3) {$y_q$};
                \node [fit=(X1)(X2-XQ)(X-label), draw] (X1-XQ) {};
                \draw[ultra thick] (X1-XQ.east) -- (T.south);
                \draw[ultra thick] (X1-XQ.east) -- (B.north);
                \draw[ultra thick] (X2-XQ.east) -- (Y2.west);
                \draw[ultra thick] (X3-XQ.east) -- (Y3.west);
                \draw[ultra thick] (XQ-2-XQ.east) -- (YQ-2.west);
                \draw[ultra thick] (XQ-1-XQ.east) -- (YQ-1.west);
                \draw (B.north) -- (YQ.west) -- (XQ);
            \end{tikzpicture}
        \caption{The interior gadget $I$ for odd $k$. The bold edges signify all possible connections are made. }\label{fig:interior-odd}
    \end{figure}
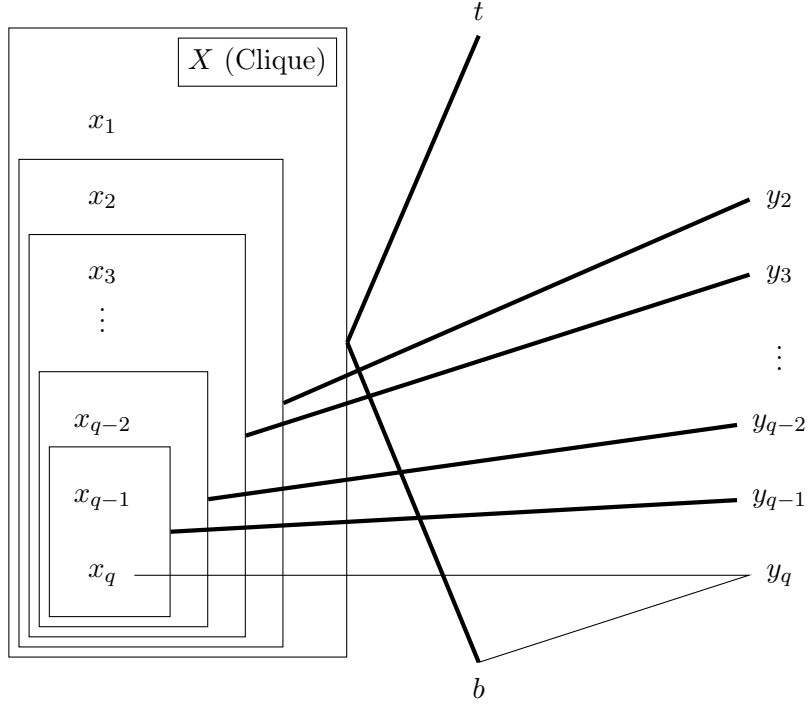

    We will prove that this graph satisfies the lemma by proving three claims.

    \begin{claim}\label{cl:induction-odd-Top-1}
        For $I$ as defined in~(\ref{eq:odd-gadget}):
        
        For all $k$-leaf roots $T$ of $I$, $m_T(t)=k\Longrightarrow m_T(b)=3$.
    \end{claim}
    \begin{proof}
        Assume that $m_T(t)=k$, then $\forall i\in [q]$, $d(t,x_i)=k$. Then for all distinct $i, j \in [q]$, by Lemma~\ref{lem:even} with $t$, $x_i$ and $x_j$, $d(x_i,x_j)$ must be even (and, thus, at most $k-1$).

        Assume $i<j<\ell$. Then $d(t,x_i)=k=\min\set{d(t,x_j),d(t,x_\ell)}$. Furthermore, $d(y_j,x_i)>k\geq \max\set{d(y_j,x_j),d(y_j,x_\ell)}$ because $y_j$ is adjacent to $x_j$ and $x_\ell$ but not $x_i$ (nor $t$). So, by Lemma~\ref{lem:closer}, we get:
        \begin{align*}
            &d(t,x_i)+d(x_j,x_\ell)<d(t,x_\ell)+d(x_i,x_j)=d(t,x_j)+d(x_i,x_\ell)
        \end{align*} 
But because $d(t,x_i)=d(t,x_j)=d(t,x_\ell)=k$. This is true if and only if
   \begin{align*}
       d(x_j,x_\ell)<d(x_i,x_j)=d(x_i,x_\ell)
        \end{align*} 

        This implies that for every $i\in[q-1]$, there exists some integer $\lambda_i$ such that $d(x_i,x_{i+1})=d(x_i,x_{i+2})=\dots=d(x_i,x_q)=\lambda_i$. Moreover, as shown, the $\lambda_i$ must be even. Thus $k-1\geq \lambda_1> \dots > \lambda_{q-1}>2$. By definition, $2q=k-1$. In particular, there are only $q-1$ even numbers greater that $2$ and at most $k-1$. Therefore, $\lambda_i=k+1-2i$.
Specifically, we have shown that $d(x_i,x_j)=k+1-2i$, for $1\leq i<j\leq q$.
        
        Recall $d(t,x_i)=k=\min\set{d(t,x_q),d(t,b)}$ and $d(y_q,x_i)>k\geq \max\set{d(y_q,x_q),d(y_q,b)}$, for $i<q$. So, by Lemma~\ref{lem:closer}, we obtain:   
        \[d(x_q,b)+d(x_i,t)<d(x_q,x_i)+d(t,b)=d(x_q,t)+d(x_i,b).\] 

        Consider $i=1$. Recall $k\geq 5$ and $q\geq 2$. So $q\neq 1$ implying that $x_q\neq x_1$. It follows that
        \begin{align*}
            &\ d(x_1,x_q)+d(t,b)=d(x_q,t)+d(x_1,b)\\
            \Longrightarrow &\ k-1+d(t,b)=k+d(x_1,b)\\
            \Longrightarrow &\ d(t,b)=1+d(x_1,b)
        \end{align*}
        However, we must have $d(t,b)>k$ and $d(x_1,b)\leq k$. So it must be that $d(t,b)=k+1$ and $d(x_1,b)=k$.

        Next consider $i=q-1$. Because $q\geq 2$, we have $i\geq 1$ and so $x_{q-1}$ exists. Then $d(x_{q-1}, x_q) = k + 1 - 2(q - 1) = k + 3 - (k - 1) = 4$. Therefore, because $d(x_i, t)=k$ for all $1\le i\le q$, we have
        \begin{align*}
            &\ d(x_{q-1},x_{q})+d(t,b)=d(x_q,t)+d(x_{q-1},b)\\
            \Longrightarrow &\  4+k+1=k+d(x_{q-1},b)\\
            \Longrightarrow &\ d(x_{q-1},b)=5
        \end{align*}

        Finally, recall that $d(x_q,b)+d(x_{q-1},t) < d(x_q,x_{q-1})+d(t,b)$. This implies that $d(x_q,b)+ k<4+(k+1)$. 
        In particular,  $d(x_q,b) < 5$.
        Moreover, by Lemma~\ref{lem:even}, we know $d(x_q,b)+d(x_q,x_{q-1})+d(x_{q-1},b)$ is even. But $d(x_q,x_{q-1})$ is even and $d(x_{q-1},b)$ is odd.
        Hence $d(x_q,b)$ must be odd. In particular, it must be odd {\bf and} less than 5. Hence $d(x_q,b)=3$, which is what we wanted to show.
    \end{proof}

    \begin{claim}\label{cl:induction-odd-Top-2}
        For $I$ as defined in~(\ref{eq:odd-gadget}):
        
        There exists a $k$-leaf root $T_I$ of $I$ such that $m_{T_I}(t)=k$ and $m_{T_I}(b)=3$.
    \end{claim}

    \begin{proof}
        We will prove this by explicitly constructing a leaf root. Recall $k$ is odd and $k=2q+1$.     
        \begin{enumerate}
            \item Take a path of length $k+1=2q+2$ from $t$ to $b$. 
            
            \item Label the vertices along the path from $t$ to $b$ which are at distance $q+i$ of $t$ as $O_i$  for $i=1,\dots, q+1$.

            \item Add a path of length $q-i+1$ from $O_i$ to $x_i$ for $i=1,\dots,q$.

            \item Add a path of length $k-2$ from $O_{q+1}$ to $y_q$.

            \item If $k\geq 7$, add a path of length $q+i$ from $O_i$ to $y_i$ for $i=2,\dots, q-1$. (For $k=5$ we have $q=2$, so these $y_i$ do not exist.)

        \end{enumerate} 
        \begin{figure}[h!]
        \centering
        \begin{tikzpicture}
            \node[circle,draw] (T) at (-6.5,0) {$t$};
            
            \node[circle,draw] (O1) at (-4.5,0) {$O_1$};
            \node[circle,draw] (O2) at (-3,0) {$O_2$};
            \node[circle,draw] (O3) at (-1.5,0) {$O_3$};
            
            \node[circle,draw] (X1) at (-4.5,2) {$x_{1}$};
            \node[circle,draw] (X2) at (-3,2) {$x_2$};
            \node[circle,draw] (X3) at (-1.5,2) {$x_3$};
            
            \node[circle,draw] (Y2) at (-3,-2) {$y_2$};
            \node[circle,draw] (Y3) at (-1.5,-2) {$y_3$};
            
            \node (dotsX) at (0,2) {$\dots$};
            \node (dots) at (0,0) {$\dots$};
            \node (dotsY) at (0,-2) {$\dots$};
            
            \node[circle,draw] (OQ-2) at (1.5,0) {$O_{q-2}$};
            \node[circle,draw] (OQ-1) at (3,0) {$O_{q-1}$};
            \node[circle,draw] (OQ) at (4.5,0) {$O_q$};
            \node[circle,draw] (OQ+1) at (6,0) {$O_{q+1}$};
            
            \node[circle,draw] (XQ-2) at (1.5,2) {$x_{q-2}$};
            \node[circle,draw] (XQ-1) at (3,2) {$x_{q-1}$};
            \node[circle,draw] (XQ) at (4.5,2) {$x_q$};

            \node[circle,draw] (YQ-2) at (1.5,-2) {$y_{q-2}$};
            \node[circle,draw] (YQ-1) at (3,-2) {$y_{q-1}$};
            \node[circle,draw] (YQ) at (6,-2) {$y_q$};
            
            \node[circle,draw] (B) at (8,0) {$b$};

            \draw (T) -- (O1) node [midway,above] {$q+1$};
            \draw (O1) -- (O2) -- (O3);
            
            \draw (O1) -- (X1) node [midway, left] {$q$};
            \draw (O2) -- (X2) node [midway, left] {$q-1$};
            \draw (O3) -- (X3) node [midway, left] {$q-2$};

            \draw (O2) -- (Y2) node [midway, left] {$q+2$};
            \draw (O3) -- (Y3) node [midway, left] {$q+3$};
            
            \draw (OQ-2) -- (XQ-2) node [midway, left] {$3$};
            \draw (OQ-1) -- (XQ-1) node [midway, left] {$2$};
            \draw (OQ) -- (XQ) node [midway, left] {$1$};
            
            \draw (OQ-2) -- (YQ-2) node [midway, left] {$k-3$};
            \draw (OQ-1) -- (YQ-1) node [midway, left] {$k-2$};
            \draw (OQ+1) -- (YQ) node [midway, left] {$k-2$};

            \draw (OQ-2) -- (OQ-1) -- (OQ) -- (OQ+1);
            \draw (OQ+1) -- (B) node [midway, above] {1};
        \end{tikzpicture}
        \caption{The $k$-leaf root $T_I$ of the interior gadget for odd $k=2q+1$. (Recall these will be connected in series below the leaf root of the top gadget; see Figure~\ref{fig:Trees}.)}\label{fig:tree-odd1}
    \end{figure}
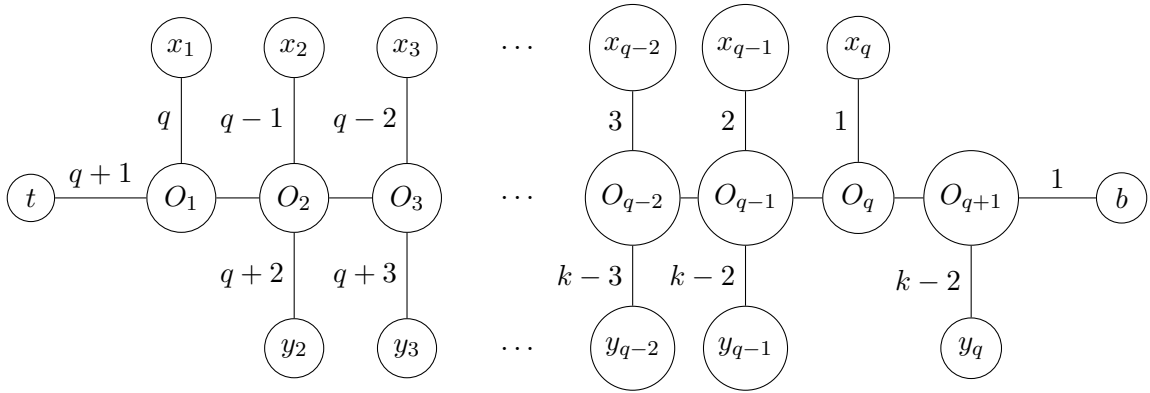
This construction is shown in Figure~\ref{fig:tree-odd1}.
        It remains to verify that this is a valid $k$-leaf root of $I$, that is, the leaf vertices at distance at most $k$ in the tree are exactly the edges of $I$.  
The required case analysis follows.
        
        \begin{itemize}
            \item The path from $t$ to $b$ has length $k+1>k$ and $t$ and $b$ are not neighbors, as desired.
            \item For $i=1,\dots, q$, the path from $t$ to $x_i$ goes through $O_i$. That is, it is the path from $t$ to $O_i$ which has length $q+i$ followed by the path from $O_i$ to $x_i$ which has length $q-i+1$. So, the path from $t$ to $x_i$ has length $(q+i)+(q-i+1)=k$.
            \item The path from $t$ to $y_q$ goes through $O_{q+1}$, so it has length $(2q+1)+(k-2)>k$.
            \item For $i = 1, \ldots, q$, the path from $b$ to $x_i$ goes through $O_q$ and $O_i$ so it has length $2+(q-i)+(q-i+1)\leq k$. In particular, for $i=q$, the path from $b$ to $x_q$ has length $3$.
            \item The path from $b$ to $y_q$ goes through $O_{q+1}$ so it has length $1+(k-2)\leq k$
            \item For $1 \leq i < j \leq q$, the path from $x_i$ to $x_j$ goes through $O_i$ and $O_j$ so it has length $(q-i+1)+ (j - i) + (q-j+1)\leq k$
            \item 
            For $i = 1, \ldots, q$ and $j = 2, \ldots, q - 1$, the path from $x_i$ to $y_j$ goes through $O_i$ and $O_j$ (possibly the same vertex) so it has length $(q-i+1)+(|i-j|)+(q+j)=k+(j-i)+|i-j|$ which is at most $k$ if and only if $i\geq j$.
            \item For $i = 1, \ldots, q$, the path from $x_i$ to $y_q$ goes through $O_i$ and $O_{q+1}$ so it has length $(q-i+1)+(q+1-i)+(k-2)$ which is equal to $k$ if $i=q$ and strictly greater than $k$ otherwise.
         \end{itemize}
         For $k=5$ the case analysis is complete.
         If $k\geq 7$ then $q\geq 3$ and so $|Y|>1$. Thus we have four more cases to verify:
         \begin{itemize}
            \item For $i = 2, \ldots, q - 1$, the path from $t$ to $y_i$  goes through $O_i$, so it has length $(q+i)+(q+i)>k$.
            \item For $i = 2, \ldots, q - 1$, the path from $b$ to $y_i$ goes through $O_q$ and $O_i$ so it has length $2+(q-i)+(q+i)>k$.
            \item For $2 \leq i < j \leq q - 1$, the path from $y_i$ to $y_j$ goes through $O_i$ and $O_j$, so it has length $(q+i)+(j - i)+(q+j)>k$.
            \item For $i = 2, \ldots, q - 1$, the path from $y_i$ to $y_q$ goes through $O_i$ and $O_{q+1}$ so it has length $(q+i)+(q+1-i)+(k-2)>k$.
         \end{itemize}
         Hence the desired $k$-leaf root exists.
    \end{proof}

    It remains to prove the final property to conclude that the claim holds when $k$ is odd:
    
    \begin{claim}\label{cl:induction-odd-Bot}
        For $I$ as defined in~(\ref{eq:odd-gadget}):
        
        There exists a $k$-leaf root $R_I$ of $I$ such that $m_{R_I}(t)=k-1$ and $m_{R_I}(b)=4$.
    \end{claim}
    
    \begin{proof}
        To construct $R_I$ we make two minor modifications to the tree $T_I$ used to prove Claim~\ref{cl:induction-odd-Top-2}. First we start with a path of length $k+2$ from $t$ to $b$. To do this we simply place $b$ at distance $2$ from $O_{q+1}$ instead of $1$. All the remaining vertices are then placed using the same process {\em except} for $x_1$ which is now at distance $q-1$ from $O_1$ instead of at distance $q$. The resultant tree is shown in Figure~\ref{fig:tree-odd2}.
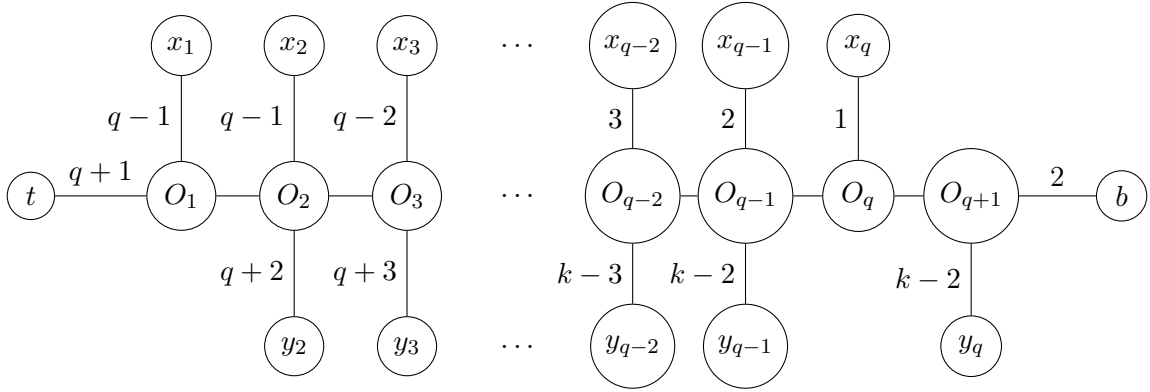
\begin{figure}[h!]
        \centering
\begin{tikzpicture}
            \node[circle,draw] (T) at (-6.5,0) {$t$};
            
            \node[circle,draw] (O1) at (-4.5,0) {$O_1$};
            \node[circle,draw] (O2) at (-3,0) {$O_2$};
            \node[circle,draw] (O3) at (-1.5,0) {$O_3$};
            
            \node[circle,draw] (X1) at (-4.5,2) {$x_{1}$};
            \node[circle,draw] (X2) at (-3,2) {$x_2$};
            \node[circle,draw] (X3) at (-1.5,2) {$x_3$};
            
            \node[circle,draw] (Y2) at (-3,-2) {$y_2$};
            \node[circle,draw] (Y3) at (-1.5,-2) {$y_3$};
            
            \node (dotsX) at (0,2) {$\dots$};
            \node (dots) at (0,0) {$\dots$};
            \node (dotsY) at (0,-2) {$\dots$};
            
            \node[circle,draw] (OQ-2) at (1.5,0) {$O_{q-2}$};
            \node[circle,draw] (OQ-1) at (3,0) {$O_{q-1}$};
            \node[circle,draw] (OQ) at (4.5,0) {$O_q$};
            \node[circle,draw] (OQ+1) at (6,0) {$O_{q+1}$};
            
            \node[circle,draw] (XQ-2) at (1.5,2) {$x_{q-2}$};
            \node[circle,draw] (XQ-1) at (3,2) {$x_{q-1}$};
            \node[circle,draw] (XQ) at (4.5,2) {$x_q$};

            \node[circle,draw] (YQ-2) at (1.5,-2) {$y_{q-2}$};
            \node[circle,draw] (YQ-1) at (3,-2) {$y_{q-1}$};
            \node[circle,draw] (YQ) at (6,-2) {$y_q$};
            
            \node[circle,draw] (B) at (8,0) {$b$};

            \draw (T) -- (O1) node [midway,above] {$q+1$};
            \draw (O1) -- (O2) -- (O3);
            
            \draw (O1) -- (X1) node [midway, left] {$q-1$};
            \draw (O2) -- (X2) node [midway, left] {$q-1$};
            \draw (O3) -- (X3) node [midway, left] {$q-2$};

            \draw (O2) -- (Y2) node [midway, left] {$q+2$};
            \draw (O3) -- (Y3) node [midway, left] {$q+3$};
            
            \draw (OQ-2) -- (XQ-2) node [midway, left] {$3$};
            \draw (OQ-1) -- (XQ-1) node [midway, left] {$2$};
            \draw (OQ) -- (XQ) node [midway, left] {$1$};
            
            \draw (OQ-2) -- (YQ-2) node [midway, left] {$k-3$};
            \draw (OQ-1) -- (YQ-1) node [midway, left] {$k-2$};
            \draw (OQ+1) -- (YQ) node [midway, left] {$k-2$};

            \draw (OQ-2) -- (OQ-1) -- (OQ) -- (OQ+1);
            \draw (OQ+1) -- (B) node [midway, above] {2};
        \end{tikzpicture}
        \caption{The $k$-leaf root $R_I$ of the interior gadget for odd $k=2q+1$. (Recall these will be connected in series above the leaf root of the bottom gadget; see Figure~\ref{fig:Trees}.)}\label{fig:tree-odd2}
    \end{figure}    
    It suffices to verify that all the vertices are still at a correct distance from $x_1$ and from~$b$.
        \begin{itemize}
            \item The distance from $x_1$ to all other vertices except $b$ has decreased by $1$ so we need to verify that the $y_i$ and $y_q$ are still at distance at least $k+1$. For $i = 2, \ldots, q - 1$, the distance from $x_1$ to $y_i$ is $(q-1)+(i-1)+(q+i)>k$ (since $i \geq 2$). The distance from $x_1$ to $y_q$ is $(q-1)+q+(k-2) = 2q + k - 3 > k$ (since $q \geq 2$).

            \item The distance from $b$ to all other vertices except $x_1$ has increased by $1$ so we need to verify that the $x_i$'s and $y_q$ are still at distance at most $k$. For $i\geq 2$, the distance from $b$ to $x_i$ is $3+(q-i)+(q-i+1)\leq k$. The distance from $b$ to $y_q$ is $2+(k-2)=k$. Moreover, all distances from $b$ to an $x_i$ vertex are now at least $4$ and not $3$ (including $x_1$, which is at distance $k=2q+1$ from $b$). In particular, the distance between $b$ and $x_q$ is exactly $4$.
        \end{itemize}
      Observing that $t$ is now at distance $k - 1$ from $x_1$ and is still at distance $k$ to its other neighbors, we get $m_{R_I}(t) = k - 1$ and $m_{R_I}(b) = 4$.  Hence the desired $k$-leaf root exists.
    \end{proof}

    We have now proven the result holds for $k$ odd. Let's now prove it for $k$ even. 
    The construction of the graph $I$ for even values of $k$ is very similar to the construction for odd values of $k$, but is slightly more intricate. Take $q=\frac{k}{2}$. The vertex set of $G$ is then
    \begin{itemize}
        \item $t$ and $b$
        \item $X=\set{x_1,\dots, x_q}$
        \item $Y=\set{y_2,\dots, y_q}$
        \item $z_1$ and $z_2$.
    \end{itemize}
    The edge set of $G$ is defined as follows:
    \begin{itemize}
        \item $(t,x_1)$ is an edge and $\forall i=2,\dots q$, $(t,x_i)$, $(b,x_i)$, $(z_1,x_i)$ and $(z_2,x_i)$ are all edges. That is, $t$ is adjacent to all vertices in $X$ while $b$, $z_1$ and $z_2$ are adjacent to all vertices in $X$ except $x_1$.
        \item For $i=1,\dots, q$, for $j=i+1,\dots q$, $(x_i,x_j)$ is an edge. That is, $X$ forms a clique.
        \item For $i=2,\dots, q$, for $j=i,\dots, q$, $(y_i,x_j)$ is an edge. 
        \item $(b,y_q)$ is an edge.
        \item $(z_1,b)$ and $(z_2,b)$ are both edges, in particular, for all $i\geq 2$, $\set{z_1,z_2,b,x_i}$ will form a diamond (with the $(z_1,z_2)$ edge being missing).
    \end{itemize}
Equivalently, it is again informative to 
define the set of edges using the neighborhoods of each vertex:
    \begin{itemize}
        \item $t$ and $x_1$ are adjacent to each other and to $X$.
        \item $b$ is adjacent to $X\setminus\{x_1\}$, to $y_q$ and to $z_1$ and $z_2$.
        \item For $i=2,\dots q$, $x_i$ is adjacent to $t$, $b$, $z_1$, $z_2$, $X\setminus \{x_i\}$ and to $y_j$ for $j=1,\dots, i$.
        \item For $i=2,\dots, q$, for $j=i,\dots, q$, $y_i$ is adjacent to $x_j$. 
        \item $y_q$ is adjacent to $x_q$ and $b$.
        \item $z_1$ and $z_2$ are adjacent to $X\setminus \{x_1\}$ and to $b$.
    \end{itemize}
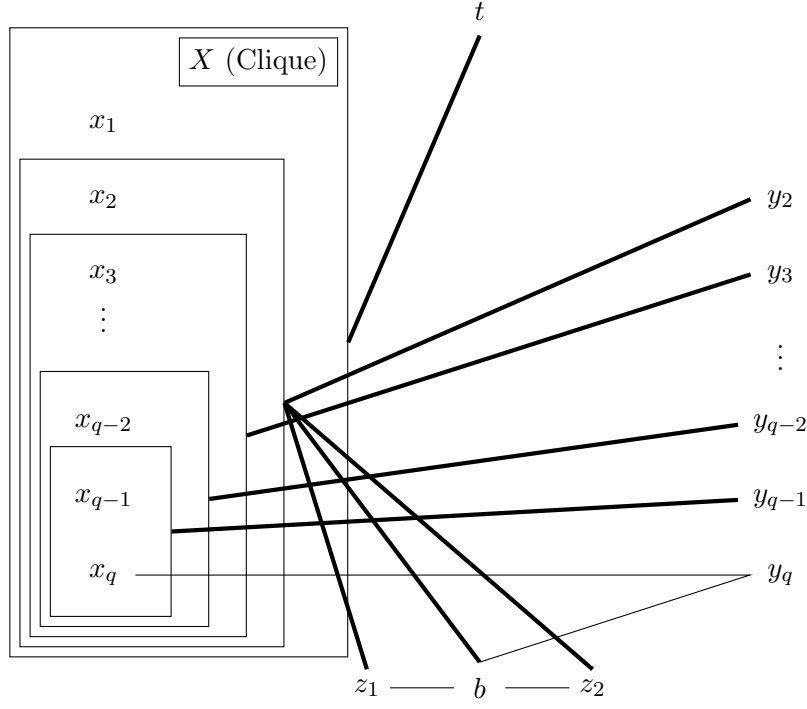
\begin{figure}[h!]
        \centering
            \begin{tikzpicture}
                \node[circle] (T) at (0,4.5) {$t$};
                \node[circle] (B) at (0,-4.5) {$b$};
                \node[circle] (X1) at (-5,3) {$x_1$};
                \node[circle] (X2) at (-5,2) {$x_2$};
                \node[circle] (X3) at (-5,1) {$x_3$};
                \node[circle] (Xdots) at (-5,0.5) {$\vdots$};
                \node[circle] (XQ-2) at (-5,-1) {$x_{q-2}$};
                \node[circle] (XQ-1) at (-5,-2) {$x_{q-1}$};
                \node[circle] (XQ) at (-5,-3) {$x_q$};
                
                \node[above right] (XQ-1-label) at (-4.5,-2) {};
                \node[above right] (XQ-2-label) at (-4,-1){};
                \node[above right] (X3-label) at (-3.5,1) {};
                \node[above right] (X2-label) at (-3,2) {};
                \node[above right,draw] (X-label) at (-4,3.5) {$X$ (Clique)};
                
                \node [align=left,fit=(XQ-1)(XQ)(XQ-1-label), draw] (XQ-1-XQ){};
                \node [align=left,fit=(XQ-2)(XQ-2-label)(XQ-1-XQ), draw] (XQ-2-XQ) {};
                \node [align=left,fit=(X3)(XQ-2-XQ)(X3-label), draw] (X3-XQ){};
                \node [align=left,fit=(X2)(X3-XQ)(X2-label), draw] (X2-XQ){};

                \node[circle] (Y2) at (4,2) {$y_2$};
                \node[circle] (Y3) at (4,1) {$y_3$};
                \node[circle] (Ydots) at (4,0) {$\vdots$};
                \node[circle] (YQ-2) at (4,-1) {$y_{q-2}$};
                \node[circle] (YQ-1) at (4,-2) {$y_{q-1}$};
                \node[circle] (YQ) at (4,-3) {$y_q$};
                \node (Z1) at (-1.5,-4.5) {$z_1$};
                \node (Z2) at (1.5,-4.5) {$z_2$};
                \node [fit=(X1)(X2-XQ)(X-label), draw] (X1-XQ) {};
                \draw[ultra thick] (X1-XQ.east) -- (T.south);
                \draw[ultra thick] (X2-XQ.east) -- (B.north);
                \draw[ultra thick] (X2-XQ.east) -- (Z1.north);
                \draw[ultra thick] (X2-XQ.east) -- (Z2.north);
                \draw[ultra thick] (X2-XQ.east) -- (Y2.west);
                \draw[ultra thick] (X3-XQ.east) -- (Y3.west);
                \draw[ultra thick] (XQ-2-XQ.east) -- (YQ-2.west);
                \draw[ultra thick] (XQ-1-XQ.east) -- (YQ-1.west);
                \draw (B.north) -- (YQ.west) -- (XQ.east);
                \draw (Z1) -- (B) -- (Z2);
            \end{tikzpicture}
        \caption{The interior gadget for even $k$.}\label{fig:interior-even}
    \end{figure}

    That is, we take $I=(V,E)$ to be defined by: 
    \begin{equation}
        \begin{aligned}
            V=&\set{t,b,z_1,z_2} \cup\left(\bigcup_{i=1}^q \set{x_i}\right)\cup\left(\bigcup_{i=2}^{q} \{y_i\}\right)\\
            E=&\set{(t,x_1)}\cup\left(\set{\bigcup_{i=2}^q \set{(t,x_i),(b,x_i),(z_1,x_i),(z_2,x_i)}}\right)\\
            &\qquad \cup \left(\bigcup_{1\leq i<j\leq q}\set{(x_i,x_j)}\right)\cup\left(\bigcup_{2\leq i\leq j\leq q}\set{(y_i,x_j)}\right)\cup\set{(b,y_q),(z_1,b),(z_2,b)}
        \end{aligned}
        \label{eq:even-gadget}
    \end{equation}

This construction is illustrated in Figure~\ref{fig:interior-even}.
    We remark that this construction only makes sense for $k\geq 4$. If $k=2$ then $Y$ is not well defined. But, while the construction makes sense for $k=4$, we will show later where it fails to work.

    Recall $z_1$ and $z_2$ form a diamond with $b$ and $x_i$ for any $i\geq 2$. Therefore, by Corollary~\ref{cor:diamond}, we have that $d(b,x_i)\neq k$ in any $k$-leaf root, in particular this is true for $i=2$ so we get $d(b,x_2)\neq k$.

    \begin{claim}\label{cl:induction-even-Top-1}
        For $I$ as defined in~(\ref{eq:even-gadget}):
        
        For all $k$-leaf roots $T$ of $I$, $m_T(t)=k\Longrightarrow m_T(b)=3$.
    \end{claim}

    \begin{proof}    
        Most of the arguments used to prove Claim~\ref{cl:induction-odd-Top-1} are still valid. Using Lemma~\ref{lem:closer} we can show that $d(x_i,x_j)=k+2-2i$. Furthermore, since $k\geq 6$, $q\geq 3$ and for $2\leq i\leq q-1$, we have $d(x_q,b)+d(x_i,t)<d(x_i,x_q)+d(t,b)=d(x_q,t)+d(x_i,b)$. 
        
        However $(x_1,b)$ is not an edge in this case so instead we must consider $i=2$. Now we get:
        \begin{align*}
            &\ d(x_q,x_2)+d(t,b)=d(x_q,t)+d(x_2,b) \\
            \Longrightarrow &\ k-2+d(t,b)=k+d(x_2,b)\\
            \Longrightarrow &\ d(t,b)=2+d(x_2,b)
        \end{align*}
        Moreover, we must have $d(t,b)>k$ and $d(x_2,b)\leq k$. So we get either $d(t,b)=k+2$ and $d(x_2,b)=k$ or $d(t,b)=k+1$ and $d(x_2,b)=k-1$. But we cannot have $d(x_2,b)=k$ as this violates Corollary~\ref{cor:diamond} when considering the diamond $(b,z_1,x_2,z_2)$. Hence, we must have $d(x_2,b)=k-1$ and $d(t,b)=k+1$. 
        
        Next, as in Claim~\ref{cl:induction-odd-Top-1}, consider $i=q-1$. Then as $d(x_{q-1}, x_q) = k + 2 - 2(q - 1) = k + 2 - (k - 2) = 4$, we get
        \begin{align*}
            &\ d(x_{q-1},x_{q})+d(t,b)=d(x_q,t)+d(x_{q-1},b)\\
            \Longrightarrow &\ 4+k+1=k+d(x_{q-1},b)\\
            \Longrightarrow &\ d(x_{q-1},b)=5
        \end{align*}
        As before, we have shown that $d(x_q,b)+d(x_{q-1},t) < d(x_q,x_{q-1})+d(t,b)$. This implies 
        $d(x_q,b)+k < 4+(k+1)$ and so $d(x_q,b)<5$. Moreover, by Lemma~\ref{lem:even}, $d(x_q,b)+d(x_q,x_{q-1})+d(x_{q-1},b)$ is even. Consequently $d(x_q,b)$ must be odd. 
        But $d(x_q,x_{q-1})=4$ is even and $d(x_{q-1},b)=5$ is odd.
        Thus $d(x_q,b)$ must be odd {\bf and} less than 5. Hence $d(x_q,b)=3$, as desired.
    \end{proof}
    \begin{claim}\label{cl:induction-even-Top-2}
        For $I$ as defined in~\ref{eq:even-gadget}:
        
        There exists a $k$-leaf root $T_I$ of $I$ such that $m_{T_I}(t)=k$ and $m_{T_I}(b)=3$.
    \end{claim}

    \begin{proof}
        We construct $I$ similarly to the proof of Claim~\ref{cl:induction-odd-Top-2}. Recall that $k=2q$.
        \begin{enumerate}
            \item Take a path of length $k+1$ from $t$ to $b$. 
            
            \item Label as $O_i$ the vertex along the path from $t$ to $b$ at distance $q-1+i$ from $t$, for $i=1,\dots, q+1$.

            \item Add a path of length $q-i+1$ from $O_i$ to $x_i$ for $i=1,\dots,q$.
                        
            \item Add a path of length $q+i-1$ from $O_i$ to $y_i$ for $i=2,\dots, q-1$.

            \item Add a path of length $k-2$ from $O_{q+1}$ to $y_q$.
            
            \item Add a path of length $q$ from $O_2$ to $z_1$ and a path of length $q$ from $O_3$ to $z_2$. (Since $k\geq 6$ we have $q\geq 3$ and, so, both $O_2$ and $O_3$ exist.)
        \end{enumerate}

    \begin{figure}[h!]
        \centering
        \begin{tikzpicture}
            \node[circle,draw] (T) at (-6.5,0) {$t$};
            
            \node[circle,draw] (O1) at (-4.5,0) {$O_1$};
            \node[circle,draw] (O2) at (-3,0) {$O_2$};
            \node[circle,draw] (O3) at (-1.5,0) {$O_3$};
            
            \node[circle,draw] (X1) at (-4.5,2) {$x_{1}$};
            \node[circle,draw] (X2) at (-3,2) {$x_2$};
            \node[circle,draw] (X3) at (-1.5,2) {$x_3$};
            
            \node[circle,draw] (Y2) at (-4,-2) {$y_2$};
            \node[circle,draw] (Y3) at (-1.5,-2) {$y_3$};
            
            \node (dotsX) at (0,2) {$\dots$};
            \node (dots) at (0,0) {$\dots$};
            \node (dotsY) at (0.5,-2) {$\dots$};
            
            \node[circle,draw] (OQ-2) at (1.5,0) {$O_{q-2}$};
            \node[circle,draw] (OQ-1) at (3,0) {$O_{q-1}$};
            \node[circle,draw] (OQ) at (4.5,0) {$O_q$};
            \node[circle,draw] (OQ+1) at (6,0) {$O_{q+1}$};
            
            \node[circle,draw] (XQ-2) at (1.5,2) {$x_{q-2}$};
            \node[circle,draw] (XQ-1) at (3,2) {$x_{q-1}$};
            \node[circle,draw] (XQ) at (4.5,2) {$x_q$};

            \node[circle,draw] (YQ-2) at (1.5,-2) {$y_{q-2}$};
            \node[circle,draw] (YQ-1) at (3,-2) {$y_{q-1}$};
            \node[circle,draw] (YQ) at (6,-2) {$y_q$};
            
            \node[circle,draw] (B) at (8,0) {$b$};
            \node[circle,draw] (Z1) at (-3,-2) {$z_1$};
            \node[circle,draw] (Z2) at (-0.5,-2) {$z_2$};
            
            \draw (T) -- (O1) node [midway,above] {$q$};
            \draw (O1) -- (O2) -- (O3);
            
            \draw (O1) -- (X1) node [midway, left] {$q$};
            \draw (O2) -- (X2) node [midway, left] {$q-1$};
            \draw (O3) -- (X3) node [midway, left] {$q-2$};

            \draw (O2) -- (Y2) node [midway, left] {$q+1$};
            \draw (O3) -- (Y3) node [midway, left] {$q+2$};
            
            \draw (OQ-2) -- (XQ-2) node [midway, left] {$3$};
            \draw (OQ-1) -- (XQ-1) node [midway, left] {$2$};
            \draw (OQ) -- (XQ) node [midway, left] {$1$};
            
            \draw (OQ-2) -- (YQ-2) node [midway, left] {$k-3$};
            \draw (OQ-1) -- (YQ-1) node [midway, left] {$k-2$};
            \draw (OQ+1) -- (YQ) node [midway, left] {$k-2$};

            \draw (OQ-2) -- (OQ-1) -- (OQ) -- (OQ+1);
            \draw (OQ+1) -- (B) node [midway, above] {1};

            \draw (O2) -- (Z1) node [midway, left] {$q$};
            \draw (O3) -- (Z2) node [midway, left] {$q$};
        \end{tikzpicture}
        \caption{The $k$-leaf root $R_I$ of the interior gadget for even $k=2q$. (Recall these will be connected in series below the leaf root of the top gadget; see Figure~\ref{fig:Trees}.)}\label{fig:tree-even1}
    \end{figure}
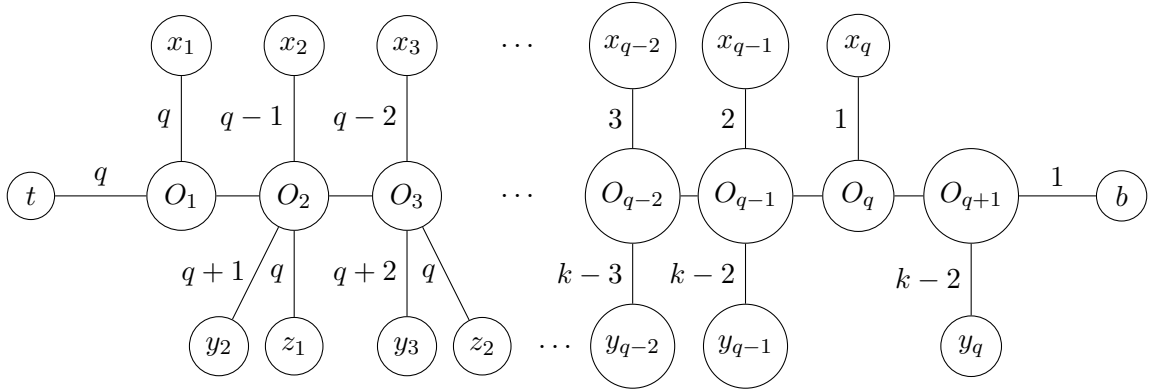
This graph $T_I$ is shown in Figure~\ref{fig:tree-even1}.
We must show $T_I$ is a $k$-leaf root of $I$. 
    The fact that these distances are valid follows from Claim~\ref{cl:induction-odd-Top-2} with the modification that $k=2q$. It remains to verify that $z_1$ and $z_2$ also have the correct neighborhoods.
        \begin{itemize}
            \item For all $i=2,\dots, q$, the vertices $t$, $x_1$, $y_q$, and $y_i$ are at distance at least $q+1$ from $O_2$ and $O_3$. Thus, they are at distance at least $2q+1 > k$ from $z_1$ and $z_2$.
            \item For $i=2,\dots,q$, the vertices $x_i$ and $b$ are both at distance at most $q$ from $O_2$ and from $O_3$. Thus they are all within distance $2q=k$ from $z_1$ and $z_2$.
            \item The path from $z_1$ to $z_2$ goes through both $O_2$ and $O_3$ so it has length $q+1+q>k$.
        \end{itemize}
        Hence the desired $k$-leaf root $T_I$ exists.
    \end{proof}

    \begin{claim}\label{cl:induction-even-Bot}
        For $I$ as defined in~\ref{eq:even-gadget}:
        
        There exists a $k$-leaf root $R_I$ of $I$ such that $m_{R_I}(t)=k-1$ and $m_{R_I}(b)=4$.
    \end{claim}
    \begin{proof}

To construct $R_I$ we make four minor modifications to the tree $T_I$ used to prove Claim~\ref{cl:induction-even-Top-2}. First we start with a path of length $k+2$ from $t$ to $b$.
To do this we simply place $b$ at distance $2$ from $O_{q+1}$ instead of $1$. 
Second, we place $x_2$ at distance $q-2$ from $O_1$ instead of at distance $q-1$. Third and fourth, we place $z_1$ and $z_2$ at distance $q$ from $O_3$ and $O_4$, respectively.
The resultant tree is shown in Figure~\ref{fig:tree-even2}.
\begin{figure}[h!]
        \centering
        
        \begin{tikzpicture}
            \node[circle,draw] (T) at (-8,0) {$t$};
            
            \node[circle,draw] (O1) at (-6,0) {$O_1$};
            \node[circle,draw] (O2) at (-4.5,0) {$O_2$};
            \node[circle,draw] (O3) at (-3,0) {$O_3$};
            \node[circle,draw] (O4) at (-1.5,0) {$O_4$};
            
            \node[circle,draw] (X1) at (-6,2) {$x_{1}$};
            \node[circle,draw] (X2) at (-4.5,2) {$x_2$};
            \node[circle,draw] (X3) at (-3,2) {$x_3$};
            \node[circle,draw] (X4) at (-1.5,2) {$x_4$};
            
            \node[circle,draw] (Y2) at (-4.5,-2) {$y_2$};
            \node[circle,draw] (Y3) at (-3.5,-2) {$y_3$};
            \node[circle,draw] (Y4) at (-1.5,-2) {$y_4$};
            
            \node (dotsX) at (0,2) {$\dots$};
            \node (dots) at (0,0) {$\dots$};
            \node (dotsY) at (0.5,-2) {$\dots$};
            
            \node[circle,draw] (OQ-2) at (1.5,0) {$O_{q-2}$};
            \node[circle,draw] (OQ-1) at (3,0) {$O_{q-1}$};
            \node[circle,draw] (OQ) at (4.5,0) {$O_q$};
            \node[circle,draw] (OQ+1) at (6,0) {$O_{q+1}$};
            
            \node[circle,draw] (XQ-2) at (1.5,2) {$x_{q-2}$};
            \node[circle,draw] (XQ-1) at (3,2) {$x_{q-1}$};
            \node[circle,draw] (XQ) at (4.5,2) {$x_q$};

            \node[circle,draw] (YQ-2) at (1.5,-2) {$y_{q-2}$};
            \node[circle,draw] (YQ-1) at (3,-2) {$y_{q-1}$};
            \node[circle,draw] (YQ) at (6,-2) {$y_q$};
            
            \node[circle,draw] (B) at (8,0) {$b$};
            \node[circle,draw] (Z1) at (-2.5,-2) {$z_1$};
            \node[circle,draw] (Z2) at (-0.5,-2) {$z_2$};
            
            \draw (T) -- (O1) node [midway,above] {$q$};
            \draw (O1) -- (O2) -- (O3) -- (O4);
            
            \draw (O1) -- (X1) node [midway, left] {$q$};
            \draw (O2) -- (X2) node [midway, left] {$q-2$};
            \draw (O3) -- (X3) node [midway, left] {$q-2$};
            \draw (O4) -- (X4) node [midway, left] {$q-3$};

            \draw (O2) -- (Y2) node [midway, left] {$q+1$};
            \draw (O3) -- (Y3) node [midway, left] {$q+2$};
            \draw (O4) -- (Y4) node [midway, left] {$q+3$};
            
            \draw (OQ-2) -- (XQ-2) node [midway, left] {$3$};
            \draw (OQ-1) -- (XQ-1) node [midway, left] {$2$};
            \draw (OQ) -- (XQ) node [midway, left] {$1$};
            
            \draw (OQ-2) -- (YQ-2) node [midway, left] {$k-3$};
            \draw (OQ-1) -- (YQ-1) node [midway, left] {$k-2$};
            \draw (OQ+1) -- (YQ) node [midway, left] {$k-2$};

            \draw (OQ-2) -- (OQ-1) -- (OQ) -- (OQ+1);
            \draw (OQ+1) -- (B) node [midway, above] {2};

            \draw (O3) -- (Z1) node [midway, left] {$q$};
            \draw (O4) -- (Z2) node [midway, left] {$q$};
        \end{tikzpicture}
        \caption{The $k$-leaf root $R_I$ of the interior gadget for even $k=2q$. (Recall these will be connected in series above the leaf root of the bottom gadget; see Figure~\ref{fig:Trees}.)}\label{fig:tree-even2}
    \end{figure}
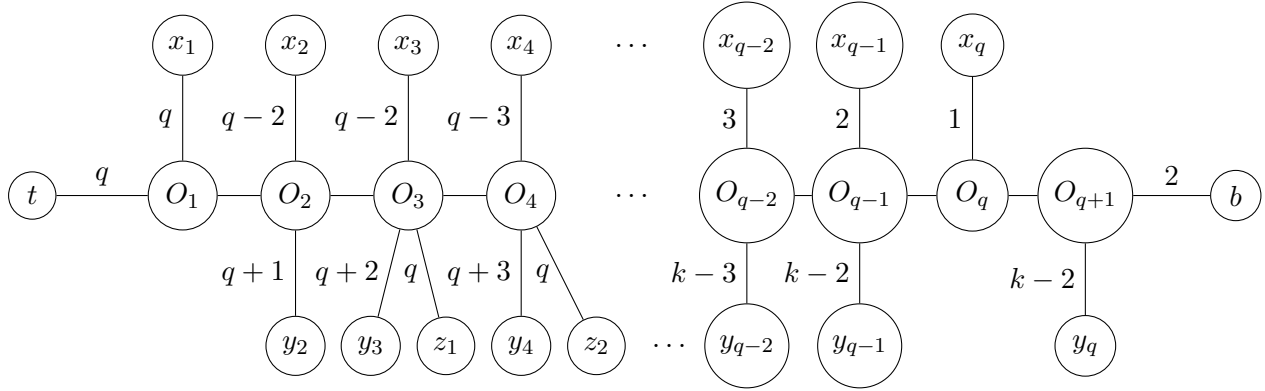

    We emphasize that this construction is \textbf{NOT} possible for $k=4$. This is because then $q=2$ and $O_4$, the vertex $O_4$ does not exist. Thus this construction applies for even $k\ge 6$.

        Finally, as in the proof of Claim~\ref{cl:induction-odd-Bot}, these changes do not change the neighborhood of the displaced vertices $b$ and $x_2$, $z_1$ and $z_2$.
        Hence the desired $k$-leaf root exists.
    \end{proof}

This completes the proof of the lemma.
\end{proof}
With our three critical lemmas proven, the main theorem now holds
by the method shown in Section~\ref{sec:proof}.

\section{Linear Leaf Powers}\label{sec:linear}

A {\em caterpillar} is a graph which has a central path and a set of leaves whose neighbor is on the central path. A graph is said to be a \textit{linear leaf power} if it has a leaf root which is the subdivision of a caterpillar. Such a leaf root is called a \textit{linear leaf root}~\cite{Bergougnoux}.

Our results apply not only to general leaf powers but also to this variant. Indeed, even if we restrict to having a subdivision of a caterpillar as a leaf root, it is impossible to get a simple forbidden
subgraph characterization of linear $k$-leaf powers for $k \geq 5$.

\begin{theorem}\label{thm:caterpillar}
    For $k\geq 5$, the set of linear $k$-leaf powers cannot be written as the set of strongly chordal graphs which are $\F_k$-free where $\F_k$ is a finite set of graphs.
\end{theorem}

As in the general case, we prove this using three gadgets. However, we must now add a condition to ensure that merging the gadgets preserves having a subdivision of a caterpillar as a leaf root.

\begin{lemma}\label{lem:top-caterpillar}
For all $k\geq 5$ there exists a gadget graph $\T$ that contains a vertex $t\in V(\T)$ such that:
\begin{enumerate}
        \item For any linear $k$-leaf root $T$ of $\T$, $m_{T}(t) = 3$.
        \item There exists a linear $k$-leaf root $T_{\T}$ of  $\T$ where $t$ is a neighbor of the last node of the central path.
    \end{enumerate}
\end{lemma}

\begin{lemma}\label{lem:bottom-caterpillar}
For all $k\geq 5$ there exists a gadget graph $\B$ that contains a vertex $b \in V(\B)$ such that:
\begin{enumerate}
        \item For any linear $k$-leaf root $T$ of $\B$, $m_T(b) \leq k-1$.
        \item There exists a linear $k$-leaf root $T_{\B}$ such that $m_{T_{\B}}(b) = k-1$ where $b$ is a neighbor of the last node of the central path.
\end{enumerate}
\end{lemma}

\begin{lemma}\label{lem:induction-caterpillar}
For all $k\geq 5$ there exists a gadget graph $I$ that contains two distinct vertices $t_I,b_I\in V(I)$ such that:
\begin{enumerate}
        \item For all linear $k$-leaf roots $T$ of $I$, $m_T(t_I)\geq k\Longrightarrow m_T(b_I)=3$.
        \item There exists a linear $k$-leaf root $T_I$ of $I$ such that $m_{T_I}(t_I)=k$ and $m_{T_I}(b_I)=3$ where $b$ and $t$ are neighbors of the first and last node of the central path respectively.
        \item There exists a linear $k$-leaf root $R_I$ of $I$ such that $m_{R_I}(t_I)=k-1$ and $m_{R_I}(b_I)=4$ where $b$ and $t$ are neighbors of the first and last node of the central path respectively.
    \end{enumerate}    
\end{lemma}

Merging the gadgets is identical to before except that we now use the condition that, in each gadget, $t$ and/or $b$ are neighbors of the extremal vertices of the central path. This means when merging the gadgets using the parent of these vertices, we merge the central paths by their endpoint to create a longer path, ensuring we produce another caterpillar subdivision.

We can verify that Lemma~\ref{lem:bottom-caterpillar} follows from Lemma~\ref{lem:bottom} and Lemma~\ref{lem:induction-caterpillar} follows from Lemma~\ref{lem:induction}, as our constructions for the interior gadget and the Bottom Gadget used for the general case satisfy the properties needed, namely the leaf root used in the proofs are subdivisions of caterpillars with the required vertices connected to the extremal vertices of the central path. On the other hand, the graph used to construct the Top Gadget does not satisfy this property\footnote{The Top Gadget for the case of general $k$-leaf powers is a caterpillar but the vertex $t$ used is not connected to an extremal vertex of the central path.}; so we need to construct a new graph for $\T$.

\begin{proof}[Proof of Lemma~\ref{lem:top-caterpillar}]
    The gadget graph $\T$ consists of a $k-1$ clique with vertices $X=\set{x_1,\dots,x_{k-1}}$, to which we add vertices $Y=\set{y_0,\dots,y_{k}}$ such that the neighborhood of $y_i$ is $X\cap \set{x_{i-1},x_{i},x_{i+1}}$ (that is, $y_i$ has 3 neighbors if $i=2,\dots,k-2$, $y_1$ and $y_{k-1}$ have two neighbors and $y_0$ and $y_{k}$ have one neighbor). This gadget is illustrated in Figure~\ref{fig:linear-top}\footnote{We remark that a more careful analysis of this family of graphs might also work for the general case.}.
  \begin{figure}[h!]
        \centering
        \begin{tikzpicture}
            \node (X1) at (180:3) {$x_1$};
            \node (X2) at (150:3) {$x_2$};
            \node (X3) at (120:3) {$x_3$};
            \node (Xdots) at (90:3) {$\dots$};
            \node (XK-3) at (60:3) {$x_{k-3}$};
            \node (XK-2) at (30:3) {$x_{k-2}$};
            \node (XK-1) at (0:3) {$x_{k-1}$};
            
            \node (Y0) at (180:5) {$y_0$};
            \node (Y1) at (165:5) {$y_1$};
            \node (Y2) at (140:5) {$y_2$};
            \node (Y3) at (125:5) {$y_3$};
            \node (Ydots) at (90:5) {$\dots$};
            \node (YK-3) at (55:5) {$y_{k-3}$};
            \node (YK-2) at (40:5) {$y_{k-2}$};
            \node (YK-1) at (15:5) {$y_{k-1}$};
            \node (YK) at (0:5) {$y_k$};

            \draw (X1) -- (X2) -- (X3) -- (XK-3) -- (XK-2) -- (XK-1) -- (X1);
            \draw (X1) -- (X3) -- (XK-2) -- (X1);
            \draw (X2) -- (XK-3) -- (XK-1) -- (X2);
            \draw (X1) -- (XK-3);
            \draw (X2) -- (XK-2);
            \draw (X3) -- (XK-1);

            \draw (Y0) -- (X1);

            \draw (Y1) -- (X1);
            \draw (Y1) -- (X2);

            \draw (Y2) -- (X1);
            \draw (Y2) -- (X2);
            \draw (Y2) -- (X3);

            \draw (Y3) -- (X2);
            \draw (Y3) -- (X3);
            \draw (Y3) -- (Xdots);

            \draw (YK-3) -- (Xdots);
            \draw (YK-3) -- (XK-3);
            \draw (YK-3) -- (XK-2);

            \draw (YK-2) -- (XK-3);
            \draw (YK-2) -- (XK-2);
            \draw (YK-2) -- (XK-1);

            \draw (YK-1) -- (XK-2);
            \draw (YK-1) -- (XK-1);

            \draw (YK) -- (XK-1);
        \end{tikzpicture}
        \caption{The top gadget for linear leaf roots.}\label{fig:linear-top}
    \end{figure}
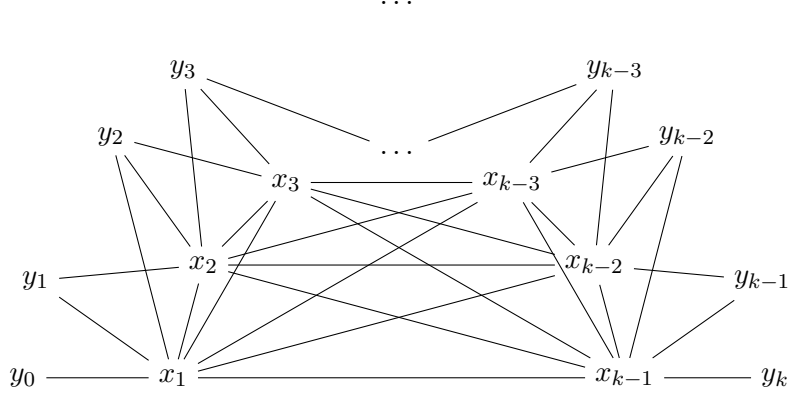

    Let $T$ be any linear $k$-leaf root of $\T$.
    For $1\leq i\leq k-1$, let $O_i$ be the vertex on the central path closest to $x_i$ in $T$. We wish to show that 
    %in any linear leaf root, 
    for $1\leq i<j\leq k-1$, $O_i\neq O_j$. For a contradiction, assume that $O_i=O_j=O$, for some $i<j$. Now consider $y_{i-1}$ and $y_{j+1}$.  Since $i<j$, $y_{i-1}$ is not a neighbor of $x_j$ and $y_{j+1}$ is not a neighbor of $x_i$ in $\T$. In $T$, every path from a leaf to $x_i$ or to $x_j$ must go through $O$. In particular, for all leaves $v$ of $T$ we have $d(x_i,O)-d(x_j,O)=d(x_i,v)-d(x_j,v)$. However, we must have $d(x_i,y_{i-1})<d(x_j,y_{i-1})$ and $d(x_i,y_{j+1})>d(x_j,y_{j+1})$. This implies that $0>d(x_i,y_{i-1})-d(x_j,y_{i-1})=d(x_i,O)-d(x_j,O)=d(x_i,y_{j+1})-d(x_j,y_{j+1})>0$, a contradiction. Thus $O_i\neq O_j$.

    Since $X$ forms a clique, all of its vertices must be within distance $k$ from one another in $T$. Moreover, they must also all be connected to a different vertex of the path. Let $\ell$ and $r$ be the endpoints of the central path. Let $O_\ell$ be the $O_i$ closest one to $\ell$ and $O_r$ the closest one to $r$. (Note that $O_\ell$ and $O_r$ are not necessarily $\ell$ and $r$ since these only take into account $X$ and not $Y$.) 
    
    Every $O_i$ must be contained in the path from $O_\ell$ to $O_r$. In particular, the distance from $x_\ell$ to $x_r$ is $k$ if and only if there is no vertex in the path which is not $O_i$ for any $i$, and both $x_\ell$ and $x_r$ are at distance one from $O_\ell$ and $O_r$, respectively. In particular, there exists a permutation $\sigma$ of $\set{1,\dots,k-1}$ such that the path from $O_\ell$ to $O_r$ is exactly $O_\ell=O_{\sigma(1)}\rightarrow O_{\sigma(2)}\rightarrow \dots\rightarrow O_{\sigma(k-2)}\rightarrow O_{\sigma(k-1)}=O_r$.
    
    Consider $x_{\sigma(1)} = x_{\ell}$ and $x_{\sigma(2)}$. By construction, one of $y_{\sigma(2)-1}$ or $y_{\sigma(2)+1}$ is not a neighbor of $x_{\sigma(1)}$ but is a neighbor of $x_{\sigma(2)}$. Let $y_{\sigma(2)\pm 1}$ denote the one which satisfies this property (or either of them if both satisfy it). 
    The path from $y_{\sigma(2)\pm 1}$ to $x_{\sigma(1)}$ and the path from $y_{\sigma(2)\pm 1}$ to $x_{\sigma(2)}$ must both go through $O_{\sigma(1)}$ or both go through $O_{\sigma(2)}$. Since $d(x_{\sigma(1)},O_{\sigma(1)})=1$, $x_{\sigma(1)}$ is closer to $O_{\sigma(1)}$ than $x_{\sigma(2)}$, therefore both paths cannot go through $O_{\sigma(1)}$ otherwise $y_{\sigma(2)\pm 1}$ would be closer to $x_{\sigma(1)}$
    than $x_{\sigma(2)}$. Moreover, since $d(x_{\sigma(1)},O_{\sigma(2)})=d(x_{\sigma(1)},O_{\sigma(1)})+d(O_{\sigma(1)}, O_{\sigma(2)}) = 2$ and $y_{\sigma(2)\pm 1}$ is  closer to $O_{\sigma(2)}$ than to $O_{\sigma(1)}$, we must have $d(x_{\sigma(2)},O_{\sigma(2)})<2$. Consequently, $d(x_{\sigma(2)},O_{\sigma(2)})= 1$.
    It immediately follows that $y_{\sigma(2)\pm 1}$ is at distance $k$ from $x_{\sigma(2)}$ and at distance $k+1$ from $x_{\sigma(1)}$. 
    
    Similarly, we can show that $x_{\sigma(1)}$ has at least one neighbor, denoted $y_{\sigma(1)\pm 1}$, which is not a neighbor of $x_{\sigma(2)}$. This implies that its distance to $O_{\sigma(1)}$ is $k-1$. However, all $x_i$, for $i\neq \sigma(1)$, are at distance at least $2$ from $O_{\sigma(1)}$. So there exists $y_{\sigma(1)\pm 1}$ which has no neighbor in $X$ except $x_{\sigma(1)}$. 

    The same argument shows that $y_{\sigma(k-1)\pm 1}$ has no neighbor in $X$ except $x_{\sigma(k-1)}$.
    But the only vertices in $Y$ of degree 1 are $y_0$ and $y_{k}$. Therefore, either $1=\sigma(1)$ or $1=\sigma(k-1)$. In either case $x_1$ is at distance exactly $3$ from another (leaf) vertex, namely, either $x_{\sigma(2)}$ or $x_{\sigma(k-2)}$.  Next observe that no two distinct leaves of $T$ can be at distance $2$. This is because  no two vertices of $\T$ have the same set of neighbors. Hence, taking $t=x_1$, we obtain $m_T(t)=3$.

    It remains to find a linear $k$-leaf root satisfying the second property in the lemma. 
    We build this tree $T_{\T}$ as follows:
    \begin{enumerate}
        \item Take a path of length $k$ from $O_1$ to $O_{k-1}$.
        \item Label the vertices along the path from $x_1$ to $x_{k-1}$ which are at distance $i$ from $x_1$ as $O_i$ for $i=1,\dots,k-1$.
        \item Add a path of length $1$ from $O_i$ to $x_i$ for $i=2,\dots,k-2$.
        \item Add a path of length $k-2$ from $O_i$ to $y_i$ for $i=1,\dots,k-1$.
        \item Add a path of length $k-1$ from $O_1$ to $y_0$ and from $O_{k-1}$ to $y_k$.
    \end{enumerate}

    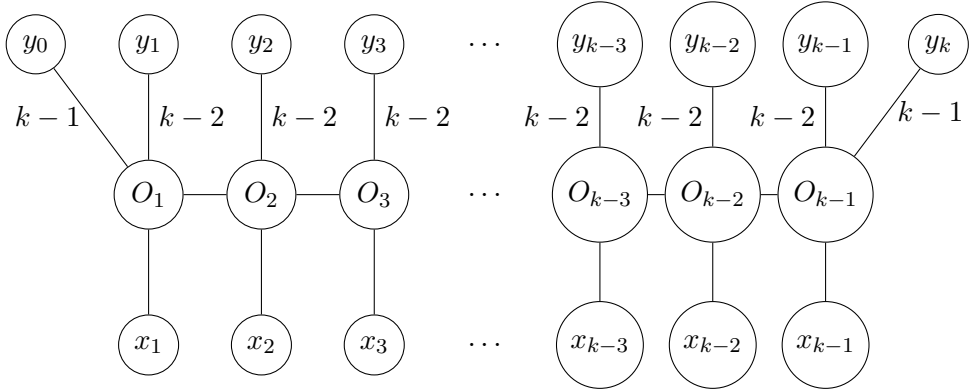
\begin{figure}[h!]
        \centering
            \begin{tikzpicture}
                \node[circle,draw] (O1) at (-4.5,0) {$O_1$};
                \node[circle,draw] (O2) at (-3,0) {$O_2$};
                \node[circle,draw] (O3) at (-1.5,0) {$O_3$};
                
                \node[circle,draw] (X1) at (-4.5,-2) {$x_{1}$};
                \node[circle,draw] (X2) at (-3,-2) {$x_2$};
                \node[circle,draw] (X3) at (-1.5,-2) {$x_3$};
                    
                \node[circle,draw] (Y0) at (-6,2) {$y_{0}$};
                
                \node[circle,draw] (Y1) at (-4.5,2) {$y_{1}$};
                \node[circle,draw] (Y2) at (-3,2) {$y_2$};
                \node[circle,draw] (Y3) at (-1.5,2) {$y_3$};
                
                \node (dotsX) at (0,2) {$\dots$};
                \node (dots) at (0,0) {$\dots$};
                \node (dotsY) at (0,-2) {$\dots$};
                    
                \node[circle,draw] (OK-3) at (1.5,0) {$O_{k-3}$};
                \node[circle,draw] (OK-2) at (3,0) {$O_{k-2}$};
                \node[circle,draw] (OK-1) at (4.5,0) {$O_{k-1}$};
                
                \node[circle,draw] (XK-3) at (1.5,-2) {$x_{k-3}$};
                \node[circle,draw] (XK-2) at (3,-2) {$x_{k-2}$};
                \node[circle,draw] (XK-1) at (4.5,-2) {$x_{k-1}$};
                
                \node[circle,draw] (YK-3) at (1.5,2) {$y_{k-3}$};
                \node[circle,draw] (YK-2) at (3,2) {$y_{k-2}$};
                \node[circle,draw] (YK-1) at (4.5,2) {$y_{k-1}$};
                \node[circle,draw] (YK) at (6,2) {$y_{k}$};
                
                \draw (O1) -- (O2) -- (O3);
                \draw (OK-3) -- (OK-2) -- (OK-1);
                
                \draw (O1) -- (X1);
                \draw (O2) -- (X2);
                \draw (O3) -- (X3);
        
                \draw (OK-3) -- (XK-3);
                \draw (OK-2) -- (XK-2);
                \draw (OK-1) -- (XK-1);
                
                \draw (O1) -- (Y0) node [midway, left] {$k-1$};
                \draw (O1) -- (Y1) node [midway, right] {$k-2$};
                \draw (O2) -- (Y2) node [midway, right] {$k-2$};
                \draw (O3) -- (Y3) node [midway, right] {$k-2$};
        
                \draw (OK-3) -- (YK-3) node [midway, left] {$k-2$};
                \draw (OK-2) -- (YK-2) node [midway, left] {$k-2$};
                \draw (OK-1) -- (YK-1) node [midway, left] {$k-2$};
                \draw (OK-1) -- (YK) node [midway, right] {$k-1$};
            \end{tikzpicture}
        \caption{The linear $k$-leaf root $T_{\T}$ for the top gadget.}
        \label{fig:caterpilar-Leaf-Root}
    \end{figure}
    This tree is shown in Figure~\ref{fig:caterpilar-Leaf-Root}.
    It is easy to verify that this tree does induce $X$ to form a clique in its $k$-leaf power graph, and that $y_i$ is only adjacent to $\set{x_{i-1},x_i,x_{i+1}}\cap X$. Moreover, taking $t = x_1$ again yields the desired result.
\end{proof}
We remark that this result does not immediately imply that there is no characterization for the entire class of linear leaf powers using chordal graphs and a finite number of forbidden induced subgraphs.
We have proven that such a characterization is impossible for each $k\geq 5$, but not necessarily for the union over all $k$. It was proved by Bergougnoux et al.~\cite{Bergougnoux} that linear leaf powers are also co-threshold tolerance. Further work on this could include verifying whether co-threshold tolerance graphs can be characterized using a finite number of obstructions.
Furthermore, it is worth noting that, despite Lemma~\ref{thm:caterpillar}, linear leaf powers are recognizable in polynomial time~\cite{Bergougnoux}. 

\section{Acknowledgements}

We would like to thank the Montreal Game Theory Workshop for helping bring together our research group.

\section{Conclusion}
% \ml{[ML: It might be nice to leave open problems for the interested readers.  
% However, you've solved everything now :p  
% So what remains?  I can think of some general directions but there may be better ideas - it'd be good to think of a specific open problem/conjecture, although I cannot think of one specific at the moment]}
We have shown that $k$-leaf powers require a deeper characterization than strong chordality with a finite set of forbidden induced subgraphs. Several directions  to explore remain in order to gain a more comprehensive understanding of $k$-leaf power graphs.
First, is it possible to construct and/or characterize minimal, strongly chordal graphs that are not $k$-leaf powers?  We were able to construct graphs $H_n$ that \emph{contain} such minimal examples as induced subgraphs; but we did not construct those examples explicitly.
Following this line of reasoning, it may be possible to characterize $k$-leaf powers as strongly chordal graphs that also forbid an additional infinite, but easy-to-describe family of forbidden subgraphs. A famous example of this are interval graphs, which are the chordal graphs containing no asteroidal triples~\cite{LB62}.

Second, are there relevant subclasses of $k$-leaf powers that can be characterized by strong chordality and a finite set of forbidden induced subgraphs?  For example, the $k$-leaf powers whose $k$-leaf roots admit a subdivision of a star should be easy enough to characterize. What about subdivisions of a tree with a small number, say two or three, of non-leaf vertices?   
One may also consider the $k$-leaf powers of caterpillars (not subdivided). Based on the midpoint arguments of Brandstädt et al.~\cite[Theorem 6]{BRANDSTADT2010897}, it would appear that, for even $k$, these coincide with the unit interval graphs whose intervals have length $k-2$ and integer endpoints.  Such (twin-free) graphs were shown to admit a finite set of forbidden induced subgraphs in~\cite{duran2015unit}. If this characterization extends to caterpillar $k$-leaf powers, this would show that taking subdivisions is necessary for our result on caterpillar graphs. 
It may also be interesting to characterize $k$-leaf powers for other graph classes that are known to be contained in $\L$; for instance, $k$-leaf powers that are also ptolemaic graphs, interval graphs, rooted directed path graphs, and others.

%[maybe talk about PCGs that use an interval whose max is at most $k$?]

%[setup ideas for the $k$-TG generalizations?]

% \printbibliography

\bibliographystyle{plain}
\bibliography{biblio}

\end{document}